\documentclass[12pt,reqno]{NumPDEsArticle}

\usepackage[utf8]{inputenc}
\usepackage[T1]{fontenc}
\usepackage[english]{babel}
\usepackage{csquotes}
\usepackage{enumitem}
\usepackage{amsmath,amssymb}
\usepackage{booktabs}
\usepackage{biblatex}
\usepackage{stmaryrd}
\usepackage[caption=false]{subfig}
\usepackage{moreverb}
\usepackage{orcidlink}
\usepackage{bm}
\usepackage[normalem]{ulem}
\usepackage{tkz-euclide}
\usepackage[caption=false]{subfig}
\usepackage{color, colortbl}
\usepackage{diagbox,graphicx}
\usepackage{subcaption}
\usepackage[export]{adjustbox}

\usepackage{NumPDEsMacros}

\newcommand{\exact}{\star}
\newcommand{\coarse}{H}
\newcommand{\fine}{h}
\newcommand{\kk}{{\underline{k}}}
\newcommand{\elll}{{\underline\ell}}
\newcommand{\jj}{{\underline j}}

\newcommand{\RRR}{\mathcal{R}}

\newcommand{\dist}{{\rm d\!l}^2}
\newcommand{\Eta}{{\rm H}}

\newcommand{\Cmesh}{\const{C}{mesh}}
\newcommand{\Ccont}{C_{\rm cnt}}

\newcommand{\llin}{\const{\lambda}{lin}}
\newcommand{\thetamark}{\const{\theta}{mark}}
\newcommand{\qalg}{\const{q}{alg}}

\newcommand{\Ceng}{C_{\rm nrg}^\star}
\newcommand{\CCeng}{C_{\rm nrg}}
\newcommand{\qnrgs}{q_{\rm nrg}^\star}
\newcommand{\qnrg}{q_{\rm nrg}}

\definecolor{best}{HTML}{2ca02c}
\definecolor{row}{HTML}{1f77b4}
\definecolor{col}{HTML}{bcbd22}
\definecolor{pyBlue}{HTML}{1f77b4}
\definecolor{pyRed}{HTML}{d62728}
\definecolor{pyGreen}{HTML}{2ca02c}
\definecolor{pyOrange}{HTML}{ff7f0e}
\definecolor{pyPurple}{HTML}{9467bd}
\definecolor{pyYellow}{HTML}{bcbd22}
\definecolor{pyGrey}{HTML}{7f7f7f}
\tikzset{reference/.style={thick,dashed}}
\tikzset{meshStyle/.style={very thin, color = black!75!white, fill=none, line 
		join=round, 
		line cap=round}}
\usetikzlibrary{shapes.geometric, calc, arrows, angles, matrix, colorbrewer}

\title[Cost-optimal FEM for nonlinear elliptic PDEs]{Unconditional 
full linear convergence and optimal complexity of adaptive iteratively 
linearized FEM for nonlinear PDE\MakeLowercase{s}}

\keywords{adaptive finite element method, optimal convergence rates,
	cost-optimality, iterative linearization, inexact solver, full linear
	convergence}

\subjclass[2020]{41A25, 65N15, 65N30, 65N55, 65Y20}

\author{Ani Miraçi\orcidlink{0000-0003-4962-9662}}
\author{Dirk Praetorius\orcidlink{0000-0002-1977-9830}}
\author{Julian Streitberger\orcidlink{0000-0003-1189-0611}}

\email{\tt ani.miraci@asc.tuwien.ac.at \rm(corresponding author)}
\email{\tt dirk.praetorius@asc.tuwien.ac.at}
\email{\tt julian.streitberger@asc.tuwien.ac.at}

\thanks{This research was funded by the Austrian Science Fund (FWF) projects
\href{https://www.fwf.ac.at/en/research-radar/10.55776/F65}{10.55776/F65} (SFB
F65 ``Taming complexity in PDE systems''),
\href{https://www.fwf.ac.at/en/research-radar/10.55776/I6802}{10.55776/I6802}
(international project I6802 ``Functional error estimates for PDEs on unbounded
domains''), and
\href{https://www.fwf.ac.at/en/research-radar/10.55776/P33216}{10.55776/P33216}
(standalone project P33216 ``Computational nonlinear PDEs'').}

\begin{document}
\maketitle

\begin{abstract}
	We propose an adaptive iteratively linearized finite element method
	(AILFEM) in the context of strongly monotone nonlinear operators with energy 
	structure in Hilbert spaces. The approach combines adaptive mesh-refinement 
	with an energy-contractive linearization scheme (e.g., the 
	Ka{\v c}anov method) and a norm-contractive algebraic solver (e.g., an 
	optimal geometric multigrid method). Crucially, a novel parameter-free 
	algebraic stopping criterion is designed and we prove that it leads to a 
	uniformly bounded number of algebraic solver steps. Unlike available results 
	requiring sufficiently small adaptivity parameters to ensure even plain 
	convergence, the new AILFEM algorithm guarantees full R-linear convergence 
	for arbitrary adaptivity parameters. Thus, unconditional convergence is
	guaranteed.
	Moreover, for sufficiently small adaptive mesh-refinement parameter
	$\theta$ and linearization-stopping parameter $\llin$, the new adaptive
	algorithm guarantees optimal complexity, i.e., optimal convergence rates
	with respect to the overall computational cost and hence time.
\end{abstract}

\section{Introduction}
\label{section:introAILFEM}

\textbf{Model problem.} Given a bounded Lipschitz domain 
$\Omega  \subset \! \mathbb{R}^d $ for
$d \in \mathbb{N}$, a scalar continuously differentiable nonlinearity 
$\mu  \colon \! \R_{\ge 0} \to \R_{\ge 0}$, and
right-hand sides $f \in  L^2(\Omega)$,
$\boldsymbol{f} \in \bigl[L^2(\Omega)\bigr]^d$, we focus on the nonlinear
elliptic PDE
\begin{equation}\label{eq:nonlinear:strongform}
	-\div \big(\mu(| \nabla u^\star |^2) \, \nabla u^\star \big)
	=
	f - \div \boldsymbol{f}
	\text{ in } \Omega
	\quad \text{subject to} \quad
	u^\star = 0
	\text{ on } \partial\Omega.
\end{equation}
We suppose that the nonlinearity satisfies a growth condition
with
$0 \! < \! \alpha \! \le \! L/3$ as
\begin{equation}\label{eq:assumption-mu}
	\alpha (t-s)
	\le
	\mu(t^2) t - \mu(s^2) s
	\le
	\frac{L}{3} \, (t-s)
	\quad \text{for all }
	t \ge s \ge 0
\end{equation}
With $\XX \coloneqq H^1_0(\Omega)$ and its topological dual $\XX'$, we consider
the problem~\eqref{eq:nonlinear:strongform} in its weak form. To this end, we
define the nonlinear operator
\begin{equation}\label{eq:nonlinearity}
	\AA
	\colon
	\XX  \to  \XX'
	\quad \text{via }
	\AA u
	\coloneqq
	\dual{\mu(|\nabla u|^2)\nabla u }{\nabla(\cdot)}_{L^2(\Omega)}
	\text{ for } u \in \XX
\end{equation}
and note that $\AA$ is {\em strongly monotone} and {\em Lipschitz continuous}.
Thus, the main theorem on strongly monotone
operators~\cite[Theorem~25.B]{zeidler} yields existence and uniqueness of the
weak solution $u^\star \in \XX$ to~\eqref{eq:nonlinear:strongform} as well as
of the exact discrete solution $u^\star_{H} \in \XX_{H}$ to a Galerkin
formulation of~\eqref{eq:nonlinear:strongform} for any discrete subspace
$\XX_{H} \subset \XX$.

Though the abstract mathematical framework of this work does not cover, e.g.,
the $p$-Laplace operator, studying the current model problem represents an
important mandatory step, before more complex nonlinear problems can be tackled.
Indeed, scalar nonlinearities arise in problems modeling a myriad of
physical phenomena, including magnetostatics, plasticity, elasticity,
gas- and hydrodynamics.
As an example, a widely used model for fluid flows in porous soil
is the Richards equation \cite{Richards31}, which in its
stationary formulation is an elliptic PDE with a scalar nonlinear diffusion.
This finds impactful applications in water management and agriculture,
but also environmental protection by modeling $\text{CO}_2$ sequestration or
control of the propagation of subterranean pollutants.

\textbf{Review of existing literature.} Note that the discretization 
of~\eqref{eq:nonlinear:strongform}
still leads to a system of nonlinear equations 
that can hardly be solved exactly.
Moreover, solutions to~\eqref{eq:nonlinear:strongform} usually exhibit
singularities that spoil the convergence of classical numerical methods.
Adaptive iteratively linearized finite element methods (AILFEMs) follow the
structure SOLVE \& ESTIMATE -- MARK -- REFINE, with the
intertwined SOLVE \& ESTIMATE module consisting of a \textit{nested} 
linearization and
algebraic solver loop. In the context of strongly monotone nonlinear problems,
the rich literature indicates a broad interest in these methods including
different linearization schemes in, e.g.,~\cite{gmz2011,gmz2012,cw2017,
	ghps2018, hw2020, hw2020:conv, ghps2021, mv2023}.
Therein, while the nonlinearity is treated iteratively, the arising linear
systems of equations are nonetheless solved exactly. 
More precisely, \cite{gmz2012} proves convergence with quasi-optimal rate for 
an adaptive algorithm assuming the exact solution of the nonlinear discrete 
system, \cite{hw2020,hw2020:conv} show plain convergence
of adaptive iteratively linearized Galerkin methods. 
Finally, \cite{ghps2018, ghps2021} show rate-optimal convergence, i.e.,
optimal convergence rates with respect to the dimension of the discrete 
ansatz space. However, cost-optimal convergence cannot be guaranteed in these 
works since the iterative solution of the arising linear FE systems is neglected. 
Indeed, it is crucial to couple discretization, 
linearization, and optimal \textit{uniformly contractive} algebraic 
solvers.
This is key for the overall algorithm to achieve \textit{optimal complexity}, 
i.e., optimal convergence rates of the error
with respect to the overall computation time in
the following sense: If $u^\star$ can be approximated by FE functions at
rate $s>0$ with respect to the number of degrees of freedom, then the algorithm
provides FE approximations that converge to $u^\star$ at rate $s$ not only 
with respect to the number of degrees of freedom, but also with respect to the 
overall computational cost.
Originating from adaptive wavelet
methods~\cite{cdd2001,
cdd2003}, this concept was later adopted in the adaptive FEM context
in~\cite{stevenson2007} for the
Poisson model problem and in~\cite{cg2012} for a Poisson eigenvalue problem.
Therein, the theory requires that the algebraic error has to be bounded by the
discretization error multiplied by a
sufficiently small solver-stopping parameter, while numerical 
experiments indicate that also larger stopping
parameters work in practice~\cite{cg2012}.

In the quasi-linear context~\eqref{eq:nonlinear:strongform},~\cite{hpsv2021} 
proposes an adaptive
lowest-order finite element method combining linearization 
(Zarantonello iteration) with an
algebraic solver (optimally preconditioned conjugate gradient method) achieving
optimal complexity. The presented adaptive algorithm steers and equibalances
the various error components (discretization, linearization, and algebraic
error) via solver-stopping criteria using parameters $\lambda_{\rm lin} $ and
$\lambda_{\rm alg}$ for the linearization and algebraic iterations,
respectively. \cite{hpsv2021} shows full R-linear convergence of the adaptive
algorithm, i.e., contraction in every step of an appropriate quasi-error
equivalent to the sum of error components, through a perturbation argument. 
This comes at the expense of
full R-linear convergence requiring \(\lambda_{\textup{alg}}\) to be 
sufficiently small with respect to
\(\lambda_{\textup{lin}}\) and \(\lambda_{\textup{lin}}\) to be small with 
respect to the Dörfler marking parameter \(\theta\).
These restrictions on the parameters are weakened in the
recent own work~\cite{bfmps2025} thanks to a summability-based 
proof strategy needing only that the algebraic solver parameter 
$\lambda_{\rm  alg}$ and the product 
$\lambda_{\rm lin} \, \lambda_{\rm  alg}$ are sufficiently small. 
The parameter choice is thus often heuristical and 
depends on the experience of the user. Importantly, even plain convergence could 
fail if the adaptivity parameters are not chosen sufficiently small.

\textbf{Contributions and novelties.} This work presents a novel 
optimal adaptive finite element algorithm with parameter-free stopping criterion 
for the algebraic solver that allows for the
cost-optimal approximation of the solution to the quasi-linear 
PDE~\eqref{eq:nonlinear:strongform}. The numerical
experiments indicate the practical success of the novel stopping criterion 
for different linearization schemes and a variety of adaptivity parameters.
More precisely, the new adaptive algorithm allows us to:
\begin{itemize}[leftmargin=6mm]
	\item drive the adaptive algorithm by fewer parameters than earlier 
	contributions, namely only $\theta$ for the mesh-refinement and 
	$\lambda_{\rm lin}$ for the linearization-stopping;
	\item enforce algorithmically the inexact energy contraction
	of iterates obtained from 
	an energy-contractive linearization coupled with a norm-contractive 
	algebraic solver;
	\item guarantee algebraic solver termination with a uniformly 
	bounded number of iterations;
	\item ensure unconditional full R-linear convergence, i.e., 
	quasi-contraction of the quasi-error for any 
	parameters \(\theta > 0\) and $\lambda_{\rm lin} > 0$, thus overcoming 
	even the weakened parameter restriction of~\cite{bfmps2025};
	\item achieve optimal complexity for a sufficiently small choice of the 
	remaining two parameters $\theta$ and $\lambda_{\rm lin}$ via a 
	perturbation argument;
	\item develop an analysis in an 
	\textit{abstract framework} where the requirements on the nonlinearity and
the constituting modules (discretization, error estimator, linearization,
algebraic solver) are clearly identified -- from this point of view, 
the results are not 
limited to the specific model problem~\eqref{eq:nonlinear:strongform}.
\end{itemize}

\textbf{Outline.} In Section~\ref{sec:setting}, we 
introduce the abstract setting to treat the
nonlinear model problem. Section~\ref{section:algorithm} introduces the
adaptive algorithm and its main properties, namely, we show that the
new stopping criterion allows to: (i) 
link the energy of the iterates with their
respective norms; (ii) guarantee a uniform number of 
algebraic solver steps. In
Section~\ref{sec:full-linear convergence}, we present our first main result on
unconditional
full R-linear convergence: in particular, this results in an equivalence of
convergence rates with respect to the degrees of freedom and with respect to
the overall computational cost. Section~\ref{sec:optimality} presents 
the
second main result of optimal complexity.
Numerical
experiments in Section~\ref{sec:numerics} confirm the theoretical results and
further investigate the performance of the novel stopping criterion
for different linearization strategies and algebraic solvers.

\section{Setting}\label{sec:setting}

\subsection{Weak formulation and inherent energy structure}

Recall $\XX = H^1_0(\Omega)$, its topological dual space
$\XX^\prime = H^{-1}(\Omega)$, and the nonlinear operator
$\AA \colon \XX \to \XX^\prime$ from~\eqref{eq:nonlinearity} associated
to~\eqref{eq:nonlinear:strongform}. With
$\enorm{\cdot} := \norm{\nabla \cdot}_{L^2(\Omega)}$, it follows
from~\eqref{eq:assumption-mu} that the operator $\AA$ is indeed strongly
monotone and Lipschitz continuous, i.e., for all $u, v, w \in \XX$ there holds
\begin{equation}\label{def:assumptions_operator}
	\alpha \, \enorm{u - v}^2
	\le
	\dual{\AA u - \AA v}{u - v}_{\XX'\times\XX}
	\quad \text{and} \quad
	\dual{\AA u - \AA v}{w}_{\XX'\times\XX}
	\le
	L \, \enorm{u - v} \, \enorm{w}.
\end{equation}
Define the right-hand side functional
$
	F(v)
	\coloneqq
	\dual{f}{v}_{L^2(\Omega)} + \dual{\boldsymbol{f}}{\nabla v}_{L^2(\Omega)}
$
for any $v\in \XX$.
The main theorem on strongly monotone operators~\cite[Theorem~25.B]{zeidler}
applies to any closed subspace $\XX_H \subseteq \XX$ and yields existence and
uniqueness of the corresponding solution $u^\star_{H} \in \XX_{H}$ to
\begin{equation}\label{eq:exact_solution_disc}
	\dual{\AA u^\star_{H}}{v_{H}}_{\XX'\times\XX}
	=
	F(v_{H})
	\quad \text{for all } v_{H} \in \XX_{H}.
\end{equation}
For $\XX = \XX_H$, note that~\eqref{eq:exact_solution_disc} is the weak
formulation of~\eqref{eq:nonlinear:strongform} and $u^\star = u^\star_H$ is the
exact solution. For any closed subspace $\XX_H \subseteq \XX$, we stress the
Céa-type quasi-optimality
\begin{equation}\label{eq:cea}
	\enorm{u^\star - u_H^\star}
	\le
	\frac{L}{\alpha} \, \enorm{u^\star - v_H}
	\quad \text{for all } v_H \in \XX_H.
\end{equation}
For the scalar nonlinearity $\mu$
from~\eqref{eq:nonlinear:strongform}--\eqref{eq:assumption-mu}, the
PDE~\eqref{eq:nonlinear:strongform} is indeed equivalent to energy
minimization: With
\begin{equation}\label{eq:energy_definition}
	\EE (v)
	\coloneqq
	\frac{1}{2} \, \int_\Omega \int_0^{|\nabla v|^2} \! \!
	\mu(t) \, {\rm d}t \, {\rm d}x - F(v),
\end{equation}
the convex \emph{energy functional} is Gâteaux differentiable and
$\mathrm{d}\EE = 0$ is indeed equivalent to~\eqref{eq:nonlinear:strongform}.
Moreover, this also yields the equivalence
(see, e.g.,~\cite[Lemma~5.1]{ghps2018})
\begin{equation}\label{eq:energy}
	\frac{\alpha}{2} \, \enorm{v_{H} - u_{H}^\star}^2
	\le
	\EE(v_{H}) - \EE(u_{H}^\star)
	\le
	\frac{L}{2} \, \enorm{v_{H} - u_{H}^\star}^2
	\quad \text{for all }
	v_{H} \in \XX_{H}.
\end{equation}
In particular, the solution $u^\star_{H} \in \XX_{H}$
of~\eqref{eq:exact_solution_disc} is
indeed the unique energy minimizer
\begin{align}\label{eq:minimization}
	\EE(u^\star_{H}) = \min_{v_{H} \in \XX_{H}} \EE(v_{H}).
\end{align}
For later use, we define the energy difference
\begin{equation}\label{eq:distance}
	\dist(v_{H},w_{H}) \coloneqq \EE(w_{H}) - \EE(v_{H})
	\quad \text{for all } v_{H}, w_{H} \in \XX_{H},
\end{equation}
and stress that
\begin{align}\label{eq:pythagoras_discrete}
	\begin{split}
		&\dist(u^\exact_{H},v_{H})
		\eqreff{eq:energy}\ge
		0,
		\quad
		\dist(v_{H},w_{H}) = - \dist(w_{H},v_{H}),
		\quad \text{and} \\
		&\dist(v_{H}, w_{H})
		=
		\dist(v_{H},z_{H}) + \dist(z_{H},w_{H})
		\quad \text{for all } v_{H},w_{H},z_{H} \in \XX_{H}.
	\end{split}
\end{align}

\subsection{Discretization}

Let $\TT_{0}$ be a conforming simplicial mesh of $\Omega$ into compact
simplices $T$. For mesh-refinement, we employ newest vertex bisection;
see~\cite{AFFKP13} for $d=1$, \cite{stevenson2008} for admissible $\TT_{0}$
when $d \ge 2$,~\cite{kpp2013} for general initial mesh $\TT_{0}$ when $d=2$,
and the recent work~\cite{dgs2023} for general $\TT_{0}$ when $d \ge 2$. We
write $\TT_h = {\tt refine}(\TT_H, \MM_H)$ for the coarsest one-level refinement
of $\TT_H$, where at least all marked elements $\MM_H \subseteq \TT_H$ have
been refined, i.e., $\MM_H \subseteq \TT_H \backslash \TT_h$. Similarly,
$\TT_h \in \mathtt{refine}(\TT_H)$ denotes that the mesh $\TT_h$ can be
obtained from $\TT_H$ by finitely many refinement steps. In particular,
$\T \coloneqq \mathtt{refine}(\TT_0)$ is the set of all meshes that can be
generated from $\TT_0$. 
Similarly, for $N \in \N$, let $\TT_h \in \T_N (\TT_H)$ be the mesh 
obtained by refining $\TT_H$ with the property $\# \TT_h - \# \TT_H \le N $.
Finally, associate to each
$\TT_H \in \T$ the conforming finite element space
\begin{equation}\label{eq:discrete_space}
	\XX_\coarse \coloneqq \set{v_\coarse \in H^1_0(\Omega) \colon v_\coarse|_T
	\text{ is a polynomial of total degree} \le p
		\text{ for all } T \in \TT_\coarse},
\end{equation}
where $p \in \mathbb{N}$ is a fixed polynomial degree, leading to
$\XX_H \subseteq \XX_h \subsetneqq \XX$ whenever $\TT_H \in \T$ and
$\TT_h \in \mathtt{refine}(\TT_H)$.

\subsection{Iterative linearization}\label{sec:linearization}

Let $\TT_H \in \T$ and $\XX_H \subset \XX$ be the corresponding
finite-dimensional (and hence closed) subspace from~\eqref{eq:discrete_space}.
To treat the arising discrete nonlinear
formulation~\eqref{eq:exact_solution_disc}, we follow the framework
of~\cite{hw2020} and consider, for any linearization point $u \in \XX$, a
symmetric bilinear form $a(u;\cdot,\cdot)$ which is uniformly elliptic and
continuous, i.e., there exist $\Cell, \Ccont > 0$ such that
\begin{equation}\label{eq:linearization:bounds}
	\Cell \, \enorm{v}^2
	\le
	a(u; v, v)
	\quad \text{and} \quad
	|a(u; v, w)|
	\le
	\Ccont \, \enorm{v}  \, \enorm{w}
	\quad
	\text{for all $u, v, w \in \XX$.}
\end{equation}
Then, for any $u_{H} \in \XX_{H}$, the Lax--Milgram lemma guarantees existence
and uniqueness of $\Phi_{H}(u_{H}) \in \XX_{H}$ solving
\begin{equation}\label{eq:linearization:exact}
	a(u_{H}; \Phi_{H}(u_{H}), v_{H})
	=
	a(u_{H}; u_{H}, v_{H})
	+ \!
	\bigl[F(v_{H})- \dual{ \AA u_{H}}{v_{H}}_{\XX' \times \XX}\bigr]
	\ \text{for all } v_{H} \in \XX_{H}.
\end{equation}
Provided that there exists $\Ceng > 0$, independent of $\XX_{H}$ and
$u_{H} \in \XX_{H}$, such that
\begin{equation}\label{eq:coercive}
	\Ceng \enorm{\Phi_{H}(u_{H}) - u_{H}}^2
	\le
	\dist(\Phi_{H}(u_{H}), u_{H})
	\text{ \, for all } u_{H} \in \XX_{H},
\end{equation}
there holds \textit{energy contraction} in the sense of
\begin{equation}\label{eq:energy-contraction}
	0
	\le
	\dist(u_{H}^\exact, \Phi_{H}(u_{H}))
	\le
	\qnrgs \, \dist(u_{H}^\exact, u_{H})
	\text{ with }
	0
	<
	\qnrgs
	\le
	1 - \frac{2\Ceng}{L} \frac{\alpha^2}{\Ccont^2} < 1;
\end{equation}
see~\cite[Theorem~2.1]{hpw2021}. In particular, \eqref{eq:energy-contraction}
provides a means to estimate the \emph{linearization error} in the sense of
\begin{equation}\label{eqq:energy-contraction}
	0
	\le
	\frac{1 - \qnrgs}{\qnrgs} \, \dist(u_{H}^\exact, \Phi_{H}(u_{H}))  \!
	\eqreff{eq:energy-contraction}
	\le  \!
	\dist(\Phi_{H}(u_{H}), u_{H}) \!
	\eqreff{eq:pythagoras_discrete}
	\le
	\dist(u_{H}^\exact, u_{H}).
\end{equation}
We recall the following examples
from~\cite{hw2020,hw2020:conv,hpw2021} that fit this framework:

\subsubsection{\bf Zarantonello (or Picard) linearization.}
For a damping parameter $\delta > 0$, 
the Zarantonello iteration~\cite{zarantonello} employs
$a(u; v,w) \coloneqq \delta^{-1} \,\dual{\nabla v}{\nabla w}_{L^2 (\Omega)}$.
Then, the iteration \(\Phi_H \colon \XX_H \to \XX\)
from~\eqref{eq:linearization:exact} reads, for all $u_H, v_H \in \XX_H$,
\begin{align}\label{eq:zarantonello_iter}
\begin{split}
\!\!\!\dual{\nabla \Phi_H(u_H)}{\! \!\nabla v_H}_{L^2 (\Omega)}
	=
	& \ \dual{\nabla u_H}{\!\!\nabla v_H}_{L^2 (\Omega)}
	\\
	& \quad + \delta \big[F(v_H) - \dual{\AA u_H}{\!v_H}_{\XX' \times \XX} \big]
	\
	\text{for all}
	\
	u_H, \!v_H \in \XX_H.
\end{split}
\end{align}
Note that~\eqref{eq:linearization:bounds} holds with
$\Cell = \delta^{-1} = \Ccont$.
For $0 < \delta < 2/L$, there
holds~\eqref{eq:coercive} with $\Ceng = \delta^{-1} - L/2$;
see~\cite[Proposition~2.3]{hpw2021}. Moreover, for $0 < \delta < 2\alpha/L^2$,
\cite[Theorem~2.2]{hw2020} proves norm contraction
\begin{equation}
	0
	\le
	\enorm{u_H^\exact - u_H^{k+1,\exact}}
	\le
	\big[1 - \delta(2\alpha-\delta L^2)\big] \,
	\enorm{u_H^\exact - u_H^{k,\exact}}
	\quad \text{for all } k \in \mathbb{N}_0,
\end{equation}
which is used in the proof of the main theorem on strongly monotone
operators~\cite[Section~25.4]{zeidler} to show via the Banach fixed-point
theorem that~\eqref{eq:exact_solution_disc} admits a unique solution
$u^\exact_{H} \in \XX_{H}$ (even if $\AA$ has no potential).

\subsubsection{\bf{Ka{\v c}anov linearization.}}

Suppose additionally that \(\mu\) is monotonically decreasing,
i.e., $\mu'(t) \le 0$ for all $t \ge 0$.
The Ka\v{c}anov linearization~\cite{kacanov} exploits the multiplicative
structure of the nonlinearity
$
	\boldsymbol{A}(\nabla u_\coarse)
	=
	\mu(| \nabla u_\coarse |^2) \, \nabla u_\coarse
$.
Then, the iteration $\Phi_H \colon \XX_H \to \XX$
from~\eqref{eq:linearization:exact} reads
\begin{equation}\label{eq:kacanov_iter}
	\dual{\mu(| \nabla u_\coarse |^2) \, 
	\nabla \Phi(u_\coarse)}{\nabla v_H}_{L^2(\Omega)} = F (v_H)
	\quad
	\text{for all}
	\quad
	u_H, v_H \in \XX_H.
\end{equation}
Since there is no damping 
parameter $\delta$ to be fine-tuned by the user and using 
the growth condition~\eqref{eq:assumption-mu} with $s=0$ to establish the uniform
bounds
\begin{equation}\label{eq:kacanov_mu_bound}
	0 < \alpha
	\le
	\mu(| \nabla w_{H} |^2)
	\le
	L/3
	\quad \text{for all} \  w_{H} \in \XX_{H},
\end{equation} 
it holds that \eqref{eq:linearization:bounds} is fulfilled with 
$\Cell = \alpha$ and
$\Ccont = L/3$. Moreover, \eqref{eq:coercive} holds with $\Ceng = \alpha/ 2$;
see~\cite[Proposition~2.4]{hpw2021}.

\subsubsection{\bf{Damped Newton method.}}\label{sec:damped_newton}
As for the Kačanov linearization, suppose additionally that \(\mu\) is monotonically decreasing,
i.e., $\mu'(t) \le 0$ for all $t \ge 0$.
Within this framework, also a damped Newton
iteration (see, e.g.,~\cite{Deuflhard_04_Newton,hw2020}) can be applied. For a
damping parameter $\delta >0$, the iteration $\Phi_H \colon \XX_H \to \XX$
from~\eqref{eq:linearization:exact} then reads, for all \(u_H, v_H \in \XX_H\),
\begin{equation}\label{eq:newton_iter}
	\dual{[\mathrm{d}\AA u_H]\Phi_H(u_H)}{v_H}_{\XX' \times \XX}
	=
	\dual{[\mathrm{d}\AA u_H]u_H}{v_H}_{\XX' \times \XX}
	+ \delta \, \big[F(v_H) -\dual{\AA u_H }{v_H}_{\XX' \times \XX} \big],
\end{equation}
where
$
	\dual{[\mathrm{d}\AA u] v}{w}_{\XX' \times \XX}
	=
	\int_\Omega 2 \mu'(|\nabla u|^2) (\nabla u \cdot \nabla v)
	(\nabla u \cdot \nabla w) {\rm d}x
	+ \int_\Omega \mu(|\nabla u|^2) (\nabla v \cdot \nabla w) {\rm d}x
$
for all $u,v,w \in \XX$ denotes the Fréchet derivative of $\AA$.

Note that for $\delta = 1$, we obtain the standard Newton iteration.
A simple monotonicity study of the functions
$ \mu(t^2)t - \alpha t$ and $ \mu(t^2)t -  Lt/3$
together with the growth condition~\eqref{eq:assumption-mu}
allows to obtain the uniform bounds
\begin{align}\label{eq:newton_mu_bound}
	\alpha \le
	2\mu' (| \nabla w_\ell |^2)| \nabla w_\ell |^2
	+ \mu (| \nabla w_\ell |^2) \le L/3
	\quad \text{for all} \  w_\ell \in \XX_\ell.
\end{align}
In particular, this yields the uniform bounds

\begin{equation*}
	\alpha \,  \enorm{v}^2
	\le
	\dual{[\mathrm{d}\AA u ]v}{v}_{\XX' \times \XX}
	\ \text{and} \
	|\dual{[\mathrm{d}\AA u ]v}{w}_{\XX' \times \XX}|
	\le
	\frac{L}{3} \, \enorm{v}  \, \enorm{w}
	\ \text{for all $u, v, w \in \XX$.}
\end{equation*}
Given $ 0 < \delta_{\min}
\le \delta < \delta_{\max} = 2 \, \alpha / L$,
then~\eqref{eq:linearization:bounds} holds with
$\Cell = \alpha / \delta_{\max}$ and $\Ccnt =
L / (3 \, \delta_{\min})$.
Moreover, \eqref{eq:coercive} holds with
$\Ceng = \alpha / \delta_{\max} - L/2$; see \cite[Proposition~2.5]{hpw2021}.

\subsection{Algebraic solver}\label{sec:algebraic_solver}
The symmetric positive definite system arising
from~\eqref{eq:linearization:exact} is, in practice, too demanding to be solved
exactly. Therefore, we require the use of a \textsl{norm-contractive} algebraic 
solver, examples include optimally preconditioned conjugate gradient methods 
such as~\cite{cg2012,hpsv2021} or optimal multigrid methods 
as~\cite{wz2017,imps2022}. Denoting $k>0$ the linearization
counter, $j>0$ the algebraic solver counter, and $u_{H}^{k,0} \in \XX_{H}$ as
initial guess, one aims to approximate the unavailable
$u_{H}^{k,\star} \coloneqq \Phi_{H} (u_{H}^{k,0})$ by iteratively computing
$u_{H}^{k,j} = \Psi_{H}(u_{H}^{k,\star};u_{H}^{k,j-1})$, with the algebraic
solver step represented by the iteration mapping
$\Psi_{H}(u_{H}^{k,\star};\cdot) \colon \XX_{H} \to \XX_{H}$. We suppose that
there exists $0 < \qalg < 1$ independent of $\XX_H$ and $u_H^{k,0}$, such that
\begin{equation}\label{eq:alg:sol:contraction}
	\enorm{u_{H}^{k,\star} - u_{H}^{k,j}}
	\le
	\qalg \, \enorm{u_{H}^{k,\star} - u_{H}^{k,j-1}}
	\quad \text{for all } j > 0.
\end{equation}
In particular, the algebraic contraction~\eqref{eq:alg:sol:contraction} and the
triangle inequality imply that
\begin{equation}\label{eq2:alg:sol:contraction}
	\frac{1 - \qalg}{\qalg} \, \enorm{u_{H}^{k,\star} - u_{H}^{k,j}}
	\le
	\enorm{u_{H}^{k,j} - u_{H}^{k,j-1}}
	\le
	(1 + \qalg) \, \enorm{u_{H}^{k,\star} - u_{H}^{k,j-1}}.
\end{equation}
Consequently, $\enorm{u_{H}^{k,j} - u_{H}^{k,j-1}}$ provides a means to
estimate the {\em algebraic error} $\enorm{u_{H}^{k,\star} - u_{H}^{k,j}}$.
Admittedly, the strongest assumption is that the contraction factor $\qalg$
of~\eqref{eq:alg:sol:contraction} does not depend on the linearization point.
However, this is clear for the Zarantonello iteration (as $a(u; \cdot,\cdot)$
in~\eqref{eq:zarantonello_iter} is independent of $u$) and for the Ka{\v c}anov
iteration, where~\eqref{eq:kacanov_mu_bound} controls the smallest and largest
eigenvalue of the matrix $\mu(\cdot) I$ in~\eqref{eq:kacanov_iter}.

\subsection{Residual \textsl{a-posteriori} error estimator}

With the nonlinearity
$
	\boldsymbol{A}(\nabla v_\coarse)
	=
	\mu(| \nabla v_\coarse |^2) \, \nabla v_\coarse
$,
let \(\jump{\cdot}\) denote the jump over $(d-1)$-dimensional faces of $\TT_H$
and consider the residual error estimator~$\eta_\coarse(\cdot)$ defined, for
\(T \in \TT_\coarse\) and \(v_\coarse \in \XX_\coarse\), by
\begin{equation}\label{eq:definition_eta}
	\! \! \!
	\eta_\coarse(T\!, \! v_\coarse)^2 \!
	\coloneqq  \!
	|T|^{2/d }
	\Vert
	f \! + \! \div(\boldsymbol{A} (\nabla  v_\coarse)  \! - \! \boldsymbol{f})
	\Vert_{L^2(T)}^{2}
	\!	+ \!
	|T|^{1/d }
	\Vert
	\jump{(\boldsymbol{A} (\nabla v_\coarse) \! - \! \boldsymbol{f}) \cdot n}
	\Vert_{L^2(\partial T \cap \Omega)}^{2}. \! \! \! \! \!
\end{equation}
To abbreviate notation, we define, for all $\UU_\coarse \subseteq \TT_\coarse$
and all $v_\coarse \in \XX_\coarse$,
\begin{equation}
	\eta_\coarse(v_\coarse)
	\coloneqq
	\eta_\coarse(\TT_\coarse, v_\coarse)
	\quad \text{with} \quad
	\eta_\coarse(\UU_\coarse, v_\coarse)
	\coloneqq
	\Big( \sum_{T \in \UU_\coarse} \eta_\coarse(T, v_\coarse)^2 \Big)^{1/2}.
\end{equation}
Then the error estimator satisfies the following properties
from~\cite[Section~10.1]{axioms}.
\begin{proposition}[axioms of adaptivity]\label{prop:axioms}
Consider the model problem~\eqref{eq:nonlinear:strongform} with the scalar
nonlinearity \(\boldsymbol{A}(\nabla u) = \mu(| \nabla u|^2) \nabla u\). Under
the assumption~\eqref{eq:assumption-mu} and for $p=1$, there exist constants
$\Cstab, \Crel, \Cdrel, \Cmon > 0$ and $0 < \qred < 1$ such that the following
properties are satisfied for any triangulation $\TT_\coarse \in \T$ and any
conforming refinement $\TT_\fine \in \T(\TT_\coarse)$ with the Galerkin
solutions $u_\coarse^\star \in \XX_\coarse$, $u_\fine^\star \in \XX_\fine$
to~\eqref{eq:exact_solution_disc} and arbitrary $v_\coarse \in \XX_\coarse$,
$v_\fine \in \XX_\fine$.
\begin{enumerate}
	\renewcommand{\theenumi}{A\arabic{enumi}}
	\bf
	\item[(A1)]\refstepcounter{enumi}\label{axiom:stability} \textit{stability}.
	\quad \rm
	$
		|\eta_\fine(\TT_\fine \cap \TT_\coarse, v_\fine)
		- \eta_\coarse(\TT_\fine \cap \TT_\coarse, v_\coarse)|
		\le
		\Cstab \, \enorm{v_\fine - v_\coarse}
	$.
	\bf
	\item[(A2)]\refstepcounter{enumi}\label{axiom:reduction} \textit{reduction}.
	\quad \rm
	$
		\eta_\fine(\TT_\fine \backslash \TT_\coarse, v_\coarse)
		\le
		\qred \, \eta_\coarse(\TT_\coarse \backslash \TT_\fine, v_\coarse)
	$.
	\bf
	\item[(A3)]\refstepcounter{enumi}\label{axiom:reliability}
	\textit{reliability}. \quad \rm
	$
		\enorm{u^\star - u_\coarse^\star}
		\le
		\Crel \,\eta_\coarse(u_\coarse^\star)
	$.

	\renewcommand{\theenumi}{A3$^{+}$}
	\bf
	\item[(A3$^+$)]\refstepcounter{enumi}
	\label{axiom:discrete_reliability}
	\it
	\textbf{discrete reliability.}
	\quad
	\rm
	$
		\enorm{u_h^\star - u_H^\star}
		\le
		\Cdrel \, \eta_H(\TT_H \backslash \TT_h, u_H^\star).
	$

	\renewcommand{\theenumi}{QM}
	\bf
	\item[(QM)]\refstepcounter{enumi}\label{eq:quasi-monotonicity}
	\textit{quasi-monotonicity}. \quad \rm
	$\eta_\fine(u_\fine^\star) \le \Cmon \, \eta_\coarse(u_\coarse^\star)$.
\end{enumerate}
The constant $\Crel$ depends only on the uniform shape regularity of all meshes
$\TT_\coarse \in \T$, the monotonicity constant $\alpha$, and the dimension
$d$; whereas $\Cstab$ and $\Cdrel$ additionally depend on the polynomial degree
$p$. The constant $\qred$ reads $\qred \coloneqq 2^{-1/(2d)}$ for
bisection-based refinement in $\R^d$ and the constant $\Cmon$ can be bounded by
$\Cmon \le  1 + \Cstab \, \Cdrel$.
\qed
\end{proposition}

\subsection{Comments} \label{sec:comments}

The analysis in the following sections applies to lowest-order FEM for the
model problem~\eqref{eq:nonlinear:strongform}--\eqref{eq:assumption-mu} from
Section~\ref{section:introAILFEM}. However, it extends beyond the restrictions
with the following understanding: In the spirit of~\cite{axioms}, our analysis
in Section~\ref{section:algorithm}--\ref{sec:optimality} below relies only on
certain assumption on the PDE and the error estimator:

\noindent
(a) The considered PDE
\begin{equation*}
	-\div \big(\boldsymbol{A} (\nabla u) \big)
	=
	F
	\text{ in } \Omega,
\end{equation*}
is complemented by appropriate boundary conditions and the nonlinearity
$\boldsymbol{A} \colon {\mathbb{R}}^d \to {\mathbb{R}}^d$ yields a strongly
monotone and Lipschitz continuous operator $ \AA \colon	\XX \to \XX'$ on
\(\XX \subseteq H^1(\Omega)\).
\noindent
(b) The operator $\AA$ has a Gâteaux-differentiable potential
$\PP \colon \XX \to \R$, i.e., it holds that $\AA = {\rm d} \PP$. In this case,
direct computation proves~\eqref{eq:energy}--\eqref{eq:distance} for the energy
\begin{equation}\label{eq:energy_pot_def}
	\EE \coloneqq \PP - F.
\end{equation}
We note that the formula~\eqref{eq:energy_pot_def} for the energy coincides
with~\eqref{eq:energy_definition} for scalar nonlinearities
$\boldsymbol{A} (\nabla u) = \mu(|\nabla u|^2)\nabla u$, while general
nonlinearities $\boldsymbol{A} \colon {\mathbb{R}}^d \to {\mathbb{R}}^d$
require additional assumptions to ensure the existence of $\PP$.

\noindent
(c) The linearization
satisfies~\eqref{eq:linearization:bounds}--\eqref{eq:coercive} in
Section~\ref{sec:linearization} (with examples given above).

\noindent
(d) The algebraic solver is contractive in the sense
of~\eqref{eq:alg:sol:contraction} in Section~\ref{sec:algebraic_solver}.

\noindent
(e) The error estimator
satisfies~\eqref{axiom:stability}--\eqref{axiom:reliability}, 
\eqref{axiom:discrete_reliability}. Currently, the proof of
stability~\eqref{axiom:stability} is only known, see \cite{gmz2012}, for scalar
nonlinearities $\boldsymbol{A}(\nabla u) = \mu(|\nabla u|^2)\nabla u$ with
growth condition~\eqref{eq:assumption-mu} and lowest-order case
$p=1$. Reduction~\eqref{axiom:reduction} holds for any continuous nonlinearity
$\boldsymbol{A} \colon {\mathbb{R}}^d \to {\mathbb{R}}^d$, and
reliability~\eqref{axiom:reliability} and discrete
reliability~\eqref{axiom:discrete_reliability} hold for any continuous
nonlinearity $\boldsymbol{A} \colon {\mathbb{R}}^d  \to {\mathbb{R}}^d$ inducing
a strongly monotone operator $ \AA \colon \XX \to \XX'$.

\section{Adaptive algorithm}
\label{section:algorithm}
For the numerical approximation of problem~\eqref{eq:exact_solution_disc}, the
Algorithm~\ref{algorithm} steers the adaptive mesh-refinement with index
$\ell$, the inexact iterative linearization with index $k$, and the contractive
algebraic solver with index $j$. In each step $(\ell,k,j)$, it yields an
approximation $u_\ell^{k,j}\in\XX_\ell$ to the unique (unavailable) exact
discrete solution $u_\ell^\star \in \XX_\ell$ to~\eqref{eq:exact_solution_disc}.
The summary of notation is presented in Table~\ref{tab_not}.
\begin{table}[htbp!]
	\centering
	\begin{tabular}{@{}cccccc@{}}
		\toprule
		& \multicolumn{2}{c}{Counter} & \multicolumn{3}{c}{Discrete Solution} \\
		\cmidrule(r){2-3} \cmidrule(l){4-6}
		& Running & Stopping & Running & Stopping & Exact \\
		\midrule
		Mesh & $\ell$ & $\underline\ell$ & $u_\ell^{\kk,\jj}$ 
		& $u_{\underline\ell}^{\kk,\jj}$ & $u_\ell^\star$ 
		from~\eqref{eq:exact_solution_disc} \\
		Linearization & $k$ & $\kk$ & $u_\ell^{k,\jj}$ & $u_{\ell}^{\kk,\jj}$ 
		& $u_\ell^{k,\star}$ from~\eqref{eq:linearized_pb} \\
		Algebraic Solver & $j$ & $\jj$ & $u_\ell^{k,j}$ & $u_\ell^{k,\jj}$ & \\
		\bottomrule
	\end{tabular}
	\vspace{0.1cm}
	\caption{\label{tab_not}Counters and discrete solutions in 
	Algorithm~\ref{algorithm}.}
\end{table}

\begin{algorithm}\label{algorithm}
\renewcommand{\theequation}{Alg\arabic{equation}}
\newcounter{tmp}
	\setcounter{tmp}{\arabic{equation}}
	\setcounter{equation}{0}
			{\bfseries Input:} Initial mesh $\TT_0$ and initial guess
			$u_0^{0,0} \coloneqq u_0^{0,\jj} = u_0^{0,\star} \in \XX_0$,
			adaptivity parameters $0 < \theta \le 1$, $1 \le \Cmark$,
			$0 < \llin$, algebraic solver parameter $0 < \rho < 1$, and
			tolerance $\tau \ge 0$.
			\\[2mm]
			Initialize $\alpha_{\rm min}> 0$, $J_{\rm max} \in \N$ with
			arbitrary values.
			\\[2mm]
			\textbf{Repeat}, for all $\ell = 0,1, \dots$, the following
			steps~\ref{alg:i}--\ref{alg:v} (adaptive loop):
			\begin{enumerate}[label = {\rm (\roman*)},leftmargin=0.8cm]
			\item\label{alg:i} \textbf{Repeat}, for all $k = 1,2, \dots$, the
			following steps~\ref{alg:a}--\ref{alg:d} (linearization loop):
			\begin{enumerate}[label = {\rm (\alph*)},leftmargin=0.7cm]
			\item\label{alg:a} Define $u_\ell^{k,0} \coloneqq u_\ell^{k-1,\jj}$.
			\item \textbf{Repeat}, for all $j= 1,2,3, \dots$, 
			the following steps~\ref{alg:I}--\ref{alg:IV} 
			(algebraic solver loop):

			\begin{enumerate}[label = {\rm (\Roman*)},leftmargin=0.5cm]
			\item \label{alg:I}  Consider the problem of finding
			$u_\ell^{k,\star} \in \XX_\ell$ such that, for all $v_\ell \in
			\XX_\ell$
			\begin{equation}
				\mkern70mu
				a(u_\ell^{k-1,\jj}; u_\ell^{k,\star},v_\ell)
				=
				a(u_\ell^{k-1,\jj}; u_\ell^{k-1,\jj}, v_\ell) + F(v_\ell) -
				\dual{\AA u_\ell^{k-1,\jj}}{v_\ell}_{\XX' \times \XX}
				\tag{Alg1} \label{eq:linearized_pb}
			\end{equation}
			and compute $u_\ell^{k,j} \approx u_\ell^{k,\star}$
			from $u_\ell^{k,j-1}$ by one step of the algebraic solver.

			\item Compute the local refinement indicators
			$\eta_\ell(T,u_\ell^{k,j})$ for all $T \in \TT_\ell$.

			\item	\textbf{If}
			$\eta_\ell(u_\ell^{k,j})
				+ \enorm{u_\ell^{k,j} - u_\ell^{k-1,\jj}}
				+ \enorm{u_\ell^{k,j} - u_\ell^{k,j-1}}
				\le \tau,$
			\hspace*{\fill}{\rm(Alg2)}\\
			\label{eq:st_crit_algorithm}
			\!\!\!\!
			\textbf{then} set $\underline \ell \coloneqq \ell$,
			$\kk \coloneqq \kk[\underline \ell] \coloneqq k$,
			$\jj \coloneqq \jj[\underline \ell, \kk] \coloneqq j$
			and terminate Algorithm~\ref{algorithm}.
			\item \label{alg:IV} Compute
			$
				\alpha_\ell^{k,j}
				\coloneqq
				\dist(u_\ell^{k,j}, u_\ell^{k-1,\jj})
				/
				\enorm{u_\ell^{k,j} - u_\ell^{k-1,\jj}}^2
			$.

			\textbf{Until} either
				$\alpha_\ell^{k,j} \ge \alpha_{\rm min} \text{ or }
				u_\ell^{k,j} = u_\ell^{k-1,\jj} \text{ or }
				\big[ \alpha_\ell^{k,j} > 0 \text{ and } j > J_{\rm max} \big]. 
				$
			\hspace*{\fill}{\rm(Alg3)}\\[-2mm]
			\label{eq:st_alg}
			\end{enumerate}
			
			\item Define $\jj \coloneqq \jj[\ell,k]\coloneqq j$.
			\item \label{alg:d} \textbf{If} $\jj[\ell,k] > J_{\rm max}$,
			\textbf{then} update $J_{\rm max} \mapsfrom \jj[\ell,k]$ and
			$\alpha_{\rm min} \mapsfrom \rho \, \alpha_{\rm min}$.
			\\[1mm]
			\textbf{Until} 
			$
				\dist(u_\ell^{k,\jj},u_\ell^{k-1,\jj})
				\le
				\llin\eta_\ell(u_\ell^{k,\jj})^2.
			$
			\hspace*{\fill}{\rm(Alg4)}\\[-2mm]
			\label{eq:st_lin}
			\end{enumerate}
			\item Define $\kk \coloneqq \kk[\ell]\coloneqq k$.
			\item Determine a set
			\begin{equation}
				\begin{aligned}
				\MM_\ell \in \M_\ell[\theta,u_\ell^{\kk,\jj}]
				&\coloneqq
				\set{
					\UU_\ell \subseteq \TT_\ell
				\colon
					\theta\, \eta_{\ell}(u_{\ell}^{\kk,\jj})^{2}
					\le  \eta_{\ell}(\UU_\ell, u_{\ell}^{\kk,\jj})^{2}
				}
				\text{ satisfying }
				\\
				\#\MM_\ell
				&\leq
				\Cmark \, \min_{\UU_\ell \in \M_\ell[\theta,u_\ell^{\kk,\jj}]} 
				\#\UU_\ell.
				\end{aligned}
				\tag{Alg5} \label{eq:doerfler}
			\end{equation}
			\item Generate
			$\TT_{\ell+1} \coloneqq \mathtt{refine}(\TT_\ell,\MM_\ell)$ and
			define
			$
				u_{\ell+1}^{0,0}
				\coloneqq
				u_{\ell+1}^{0,\jj}
				\coloneqq
				u_{\ell+1}^{0, \star}
				\coloneqq
				u_\ell^{\kk,\jj}
			$.
			\item \label{alg:v} Update counters $\ell \coloneqq \ell + 1$,
			$k \coloneqq 0$, and $j \coloneqq 0$ and continue with {\rm(i)}.
		\end{enumerate}
		\setcounter{equation}{\arabic{tmp}}
\end{algorithm}
Some remarks are in order to explain the nature of Algorithm~\ref{algorithm}.
The innermost loop (Algorithm~\ref{algorithm}(i.b) with index \(j\)) steers the
algebraic solver. Note that the exact solution $u_\ell^{k,\star}$
of~\eqref{eq:linearized_pb} is not computed but only approximated by the
iterates $u_\ell^{k,j}$. The middle loop (Algorithm~\ref{algorithm}(i) with
index \(k\)) steers the iterative linearization. Moreover, the iterative
linearization~\eqref{eq:linearized_pb} is perturbed, since instead of the
unavailable $u_\ell^{k,\star}$, we use the computed final iterate
$u_\ell^{k,\jj}$; cf.~\eqref{eq:linearization:exact}--\eqref{eq:coercive}.

We stress the new stopping
criterion~(\hyperref[eq:st_alg]{Alg3}) for the algebraic solver 
(with index \(j\)). The main idea behind this criterion 
stems from the central link of norm and energy encoded in \eqref{eq:coercive}. 
Indeed, this leads to the energy contraction \eqref{eq:energy-contraction} 
(if the exact algebraic 
solution is computed) and it is an essential ingredient for 
\textit{energy-steered} adaptive algorithms as in, e.g.,
\cite{hpw2021}. Since we now employ an 
\textit{iterative} algebraic solver, we need to \textit{algorithmically} 
enforce a similar link on the level of the inexact solver to ensure 
a similar energy contraction for the 
obtained iterates. Thus, heuristically, we iterate the 
algebraic solver until \eqref{eq:coercive} is satisfied for the 
\textit{inexact solution}. We prove in Proposition~\ref{prop:unif-steps} 
that this criterion will indeed be satisfied for a 
\emph{fixed} number of algebraic solver steps. Moreover, 
we devise the stopping criterion to be \textit{parameter-free}: 
$\alpha_\ell^{k,j}$ plays the role of the constant in the norm-energy link, 
which possibly needs to be made smaller (via multiplication by $\rho<1$) whenever 
a pre-fixed maximum algebraic solver steps $J_{\rm max}$ is reached. The latter 
is also increased when the norm-energy link is not satisfied 
(Proposition~\ref{prop:unif-steps} ensures that there is a uniform upper bound 
for $J_{\rm max}$).
Structurally, 
the stopping
criterion~(\hyperref[eq:st_lin]{Alg4}) for the iterative linearization 
(with index \(k\)) is the same as 
in~\cite{ghps2021,hpsv2021,hpw2021, bfmps2025} 
and aims to balance, via the associated \textsl{a~posteriori} estimators, 
the linearization error relative to the discretization error.
Finally, the outermost loop (with index \(\ell\)) steers the local adaptive
mesh-refinement by employing the D\"orfler criterion~\eqref{eq:doerfler} 
from~\cite{doerfler1996} used to mark
elements $T \in \MM_\ell$ for refinement. 
Moreover, note that Algorithm~\ref{algorithm}
uses \emph{nested iteration}: each loop is initialized with the
previously available approximation. This is crucial from both a 
practical and analytical standpoint as it is needed for the linear complexity 
and convergence analysis of the adaptive algorithm.

\subsection{Index set \(\QQ\) for the triple loop}

To analyze the asymptotic convergence behavior of Algorithm~\ref{algorithm} for
tolerance $\tau = 0$, we define the index set
\begin{equation} \label{eq_QQ}
	\QQ \coloneqq \set{(\ell,k,j) \in \mathbb{N}_0^3 \colon \text{index triple
	$(\ell,k,j)$ is used in Algorithm~\ref{algorithm}}}.
\end{equation}
Since Algorithm~\ref{algorithm} is sequential, the index set $\QQ$ is naturally
ordered by the lexicographic ordering. For indices
$(\ell,k,j), (\ell',k',j') \in \QQ$, we write
\begin{equation}
	(\ell,k,j)
	<
	(\ell',k',j')
	\,\,\, \stackrel{\text{def}}{\Longleftrightarrow} \,\,\,
	(\ell,k,j)
	\text{ appears earlier in Algorithm~\ref{algorithm} than }
	(\ell',k',j').
\end{equation}
With this ordering, we can define
\begin{equation*}
	|\ell,k,j|\coloneqq\#\set{(\ell',k',j')\in\QQ \colon
	(\ell',k',j')<(\ell,k,j)},
\end{equation*}
which is the {\em total step number} of Algorithm~\ref{algorithm}.
We make the following definitions, which are consistent with that of
Algorithm~\ref{algorithm}:
\begin{align*}
	\underline\ell
	&\coloneqq
	\sup\set{\ell \in \mathbb{N}_0 \colon (\ell,0,0) \ \in \QQ}
	\in \mathbb{N}_0 \cup\{\infty\},
	\\
	\kk[\ell]
	&\coloneqq
	\sup\set{k \in \mathbb{N} \colon (\ell,k,0) \, \in \QQ}
	\in \mathbb{N} \cup\{\infty\}
	\quad \text{if } (\ell,0,0) \in \QQ,
	\\
	\jj[\ell,k]
	&\coloneqq
	\sup\set{j \in \mathbb{N} \colon (\ell,k,j) \ \in \QQ}
	\in \mathbb{N} \cup \{ \infty \}
	\quad \text{if } (\ell,k,0) \in \QQ.
\end{align*}
To abbreviate notation, we use the following convention: If the mesh index
$\ell \in \mathbb{N}_0$ is clear from the context, we simply write
$\kk = \kk[\ell]$ for, e.g., $u_\ell^{\kk,j} = u_\ell^{\kk[\ell],j}$ or
$(\ell,\kk[\ell],j) = (\ell,\kk,j)$.  Similarly, we simply write
$\jj = \jj[\ell,k]$ for, e.g., $u_\ell^{k,\jj} = u_\ell^{k,\jj[\ell,k]}$ or
$(\ell,k,\jj) = (\ell,k,\jj[\ell,k])$. We observe that, for all
$(\ell,k,0) \in \QQ$, it holds $\kk[\ell] \ge 1$ and
\begin{align*}
	&\jj [\ell, k]
	=
	0
	\quad \text{for} \ k=0 \
	\text{and}
	\\
	&\jj [\ell, k]
	\ge
	1
	\quad \text{for} \
	k\ge 1.
\end{align*}
Moreover, Proposition~\ref{prop:unif-steps} below proves that
$\jj[\ell, k]  \le j_0 < \infty $ with a uniform bound $j_0 \in \mathbb{N}$. For
infinitesimal tolerance $\tau = 0$, it generically holds that $\# \QQ = \infty$.
Our analysis below covers indeed the three exclusive cases in
Algorithm~\ref{algorithm} that can occur:
\begin{enumerate}[label=(\roman*)]
	\item There holds $\elll = \infty$ and hence $\kk [\ell] < \infty$ for all
	$\ell \in \mathbb{N}_0$, i.e., infinitely many steps of mesh-refinement
	take place.
	\item There holds $\elll < \infty$ with $\kk [\elll] = \infty$, in which
	case Corollary~\ref{corollary_0tol} below yields that
	$\eta_{\elll} (u_{\elll}^\star) = 0$ and hence
	$u^\star = u_{\elll}^\star$, while $u_\elll^{k,\jj} \neq u_\elll^{\star}$
	for all $k \in \mathbb{N}_0$.
	\item The exact solution $u^\star = u_{\elll}^{\kk, \jj}$ is hit at
	Algorithm~\ref{algorithm}(i.b.III) with
	$\eta_{\underline\ell}(u_{\underline\ell}^{\kk,\jj}) = 0$.
\end{enumerate}

\subsection{Termination of algebraic solver and a~posteriori control}

We present now some important properties of Algorithm~\ref{algorithm}, starting
with the following observation in the spirit of, e.g., 
\cite[Proof of Theorem 2.6]{hw2020}. Details are found in the 
Appendix~\ref{appendix}.
%
\begin{lemma}\label{lemma:energy-identity}
	For any $v_\ell, w_\ell \in \XX_\ell$, it holds that
	\begin{align}\label{eq2:lemma:energy-identity}
		\begin{split}
			\dist(v_\ell, w_\ell)
			&
			\le
			L \, \enorm{w_\ell - u_\ell^{k-1,\jj}} \, \enorm{v_\ell - w_\ell}
			+ \frac{L}{2} \, \enorm{v_\ell - w_\ell}^2
			\\&
			\qquad
				+ a(
					u_\ell^{k-1,\jj}; u_\ell^{k,\star} - u_\ell^{k-1,\jj},
					v_\ell-w_\ell
					).
		\end{split}
	\end{align}
	Moreover, for $w_\ell = u_\ell^{k-1,\jj}$, it follows that
	\begin{align}\label{eq2:lemma:energy-identity}
		\dist(v_\ell, u_\ell^{k-1,\jj})
		&\ge
		- \frac{L}{2} \, \enorm{v_\ell-u_\ell^{k-1,\jj}}^2
		+ a(
			u_\ell^{k-1,\jj}; u_\ell^{k,\star} - u_\ell^{k-1,\jj},
			v_\ell-u_\ell^{k-1,\jj}
			),
		\\ \label{eq3:lemma:energy-identity}
		\dist(v_\ell, u_\ell^{k-1,\jj})
		&\le
		- \frac{\alpha}{2} \, \enorm{v_\ell-u_\ell^{k-1,\jj}}^2
		+ a(
			u_\ell^{k-1,\jj}; u_\ell^{k,\star} - u_\ell^{k-1,\jj},
			v_\ell-u_\ell^{k-1,\jj}
			).
	\end{align}
\end{lemma}
Since the algebraic solver contracts in
norm~\eqref{eq:alg:sol:contraction} and the linearization contracts in
energy~\eqref{eq:energy-contraction}, we derive a crucial
estimate linking norm and energy of the iterates in the following 
proposition. Moreover, we show that the innermost loop of
Algorithm~\ref{algorithm}\textrm{(i.b)} terminates after a \textit{uniformly}
bounded
number of steps.
\begin{proposition}[uniform bound on number of algebraic solvers
steps]\label{prop:unif-steps}
	Consider arbitrary parameters $0 < \theta \le 1$, $1 \le \Cmark$,
	$0 < \llin$, $0 < \rho < 1$, \(\alpha_{\min} > 0\), \(J_{\max} \in \N\),
	$\tau \ge 0$, and arbitrary \(u_0^{0,0} \in \XX_0\) in
	Algorithm~\ref{algorithm}. Then, there holds the following:
	\begin{enumerate}[label=(\roman*), font = \upshape]
		\item For any $0 < \CCeng' < \Ceng$, there
	exists $j_0' \in \N$ such that
	\begin{equation}\label{eq:norm_energy_estimate}
		\CCeng' \, \enorm{u_\ell^{k,j} - u_\ell^{k-1,\jj}}^2
		\le
		\dist(u_\ell^{k,j}, u_\ell^{k-1,\jj})
		\quad  \text{for all }
		(\ell,k,j) \in \QQ
		\text{ with }
		j \ge j_0'.
	\end{equation}
		\item There exists an index $j_0 \in \N$ such
		that $\jj[\ell,k] \le j_0$ for all $(\ell,k,0) \in \QQ$, i.e., 
		the number
		of algebraic solver steps is finite and even uniformly bounded.
		Moreover, there exists $0< \CCeng < \Ceng$ such that
	\begin{equation}\label{eq:norm_energy_estimate_finalj}
		\CCeng \, \enorm{u_\ell^{k,\jj} - u_\ell^{k-1,\jj}}^2
		\le
		\dist(u_\ell^{k,\jj}, u_\ell^{k-1,\jj})
		\quad \text{for all }
		(\ell,k,0) \in \QQ
		\text{ with }
		k \ge 1.
	\end{equation}
	\end{enumerate}
\end{proposition}
\begin{proof}
	The proof consists of two steps.

	\textbf{Step 1: Proof of norm-energy 
	estimate~\eqref{eq:norm_energy_estimate}.}
	Let $0 < \CCeng' < \Ceng$. For any $(\ell, k, \jj) \in \QQ$ and for the
	theory only, we disregard the termination index $\jj [\ell, k]$ and
	formally continue the algebraic solver for all $j > \jj[\ell,k]$ by
	defining, but not computing in practice,
	$u_\ell^{k,j} = \Psi_\ell (u_\ell^{k,j-1}) \in \XX_\ell$.
	By assumption~\eqref{eq:coercive} and
	$u_\ell^{k,\star} = \Phi_\ell(u_\ell^{k-1,\jj})$, it holds that
	\begin{equation}\label{eq:prep_est_norm}
		\Ceng \, \enorm{u_\ell^{k,\star} - u_\ell^{k-1,\jj}}^2
		\eqreff{eq:coercive}
		\le
		\dist(u_\ell^{k,\star}, u_\ell^{k-1,\jj})
		\eqreff{eq:pythagoras_discrete}
		=
		\dist(u_\ell^{k,\star}, u_\ell^{k,j})
		+ \dist(u_\ell^{k,j}, u_\ell^{k-1,\jj}).
	\end{equation}
	We apply~\eqref{eq2:lemma:energy-identity} from
	Lemma~\ref{lemma:energy-identity} for $ v_\ell = u_\ell^{k,\star}$ and $
	w_\ell = u_\ell^{k,j}$, 
	combined with \eqref{eq:linearization:bounds} to obtain
	\begin{align*}
		\begin{split}
			&\dist(u_\ell^{k,\star}, u_\ell^{k,j})
			\le
			\frac{3L}{2} \, \enorm{u_\ell^{k,\star} - u_\ell^{k,j}}^2
			+ (L + \Ccnt) \, \enorm{u_\ell^{k,\star} - u_\ell^{k-1,\jj}} \,
				\enorm{u_\ell^{k,\star} - u_\ell^{k,j}}.
		\end{split}
	\end{align*}
	Contraction of the algebraic
	solver~\eqref{eq:alg:sol:contraction} and nested iteration
	$u_\ell^{k,0} = u_\ell^{k-1,\jj}$ leads to
	\begin{align}
		\begin{split}
			\label{eq:prep_est_energy}
			\dist(u_\ell^{k,\star}, u_\ell^{k,j})
			&\eqreff*{eq:alg:sol:contraction}
			\le \,
			\Big[ 
				\frac{5L}{2} + \Ccnt
			\Big] \,
			\qalg^{j}
			\, \enorm{u_\ell^{k,\star} - u_\ell^{k-1,\jj}}^2
			\eqqcolon
			C \, \qalg^{j} \, \enorm{u_\ell^{k,\star} - u_\ell^{k-1,\jj}}^2.
		\end{split}
	\end{align}
	Since $ 0< \qalg < 1$ and $0<\CCeng'<\Ceng$, there exists a minimal
	$j_0 \in \N$ such that, for all $j\ge j_0$, holds
	$\CCeng' < \Ceng - C \, \qalg^{j} < \Ceng$. 
	By combining~\eqref{eq:prep_est_norm}--\eqref{eq:prep_est_energy}, we obtain
	\begin{equation}\label{eq:est1norm-nrg}
		0
		<
		(\Ceng - C \, \qalg^{j}) \enorm{u_\ell^{k,\star} - u_\ell^{k-1,\jj}}^2
		\le
		\dist(u_\ell^{k,j}, u_\ell^{k-1,\jj})
		\quad \text{for all} \
		j \ge j_0.
	\end{equation}
	Contraction of the algebraic solver~\eqref{eq:alg:sol:contraction} and
	nested iteration $u_\ell^{k,0} = u_\ell^{k-1,\jj}$ lead to
	\begin{equation}\label{eq:est2norm-nrg}
		\enorm{u_\ell^{k,j} - u_\ell^{k-1,\jj}}
		\le
		\enorm{u_\ell^{k,\star} - u_\ell^{k-1,\jj}}
		+ \enorm{u_\ell^{k,\star} - u_\ell^{k,j}}
		\le
		(1+ \qalg^{j}) \enorm{u_\ell^{k,\star} - u_\ell^{k-1,\jj}}.
	\end{equation}
	Choose $j'_0 \in \mathbb{N}_0$ minimal such that
	$
		0
		<
		\CCeng'
		\le
		\frac{\Ceng - C \, \qalg^{j} }{(1+ \qalg^{j})^2}
		<
		\Ceng - C \, \qalg^{j}
		<
		\Ceng
	$
	for all $j \ge j'_0$, where $j'_0 \in \mathbb{N}$ depends only on
	$L, \, \Ccnt, \, \qalg$, and $ \CCeng' < \Ceng$.
	Combining~\eqref{eq:est1norm-nrg}--\eqref{eq:est2norm-nrg} with the choice
	of $j'_0$, we are led to
	\begin{equation*}
		\CCeng' \, \enorm{u_\ell^{k,j} - u_\ell^{k-1,\jj}}^2
		\le
		\frac{\Ceng - C \, \qalg^{j} }{(1+ \qalg^{j})^2}
		\, \enorm{u_\ell^{k,j} - u_\ell^{k-1,\jj}}^2
		\le
		\dist(u_\ell^{k,j}, u_\ell^{k-1,\jj})
		\quad \text{for all} \
		j \ge j'_0.
	\end{equation*}

	\medskip
	\textbf{Step~2 (uniform bound on algebraic solver steps).}
	We argue by contradiction and assume that for any $j_0 \in \N$, there
	exists $(\ell,k,0) \in \QQ$ such that $\jj[\ell,k] > j_0$. On the one hand,
	Algorithm~\ref{algorithm}(i.d) then implies that 
	$\lim_{|\ell,k,\jj| \to \infty} J_{\rm max} [\ell, k] = \infty$ and
	$\lim_{|\ell,k,\jj| \to \infty}  \alpha_{\rm min}[\ell, k] =  0$, where we
	index $J_{\rm max} \in \N$ and $\alpha_{\rm min}>0$ by $[\ell, k] $, since
	they are potentially adapted in Algorithm~\ref{algorithm}(i.d) after the
	$j$-loop has been terminated. On the other hand, given $0<\CCeng' < \Ceng$,
	estimate~\eqref{eq:norm_energy_estimate} from (i)
	guarantees the existence of $j_0' \in \N_0$ such that
	$\CCeng' \le \alpha_\ell^{k,j}$ for all $(\ell,k,j) \in \QQ$ with
	$j \ge j_0'$. We fix $(\ell_0,k_0,\jj) \in \QQ$ such that
	$\alpha_{\rm min}[\ell, k] \le \CCeng'$ for all $(\ell,k,0) \in \QQ$ with
	$|\ell,k,0| > |\ell_0,k_0,\jj|$. However, for all $(\ell,k,0) \in \QQ$ with
	$|\ell,k,0| > |\ell_0,k_0,\jj|$, this implies that the termination of
	criterion~(\hyperref[eq:st_alg]{Alg3}) will be satisfied for some
	$j = \jj [\ell, k] \le j_0'$. In particular, this yields that
	$\jj [\ell, k] \le J_{\rm max} [\ell_0, k_0]$ and prevents that
	$J_{\rm max} [\ell, k]  \rightarrow \infty$ as
	$|\ell,k,j| \rightarrow \infty$. This contradiction concludes the proof.
\end{proof}
All iterates computed by Algorithm~\ref{algorithm} can be reliably controlled,
which motivates the choice of the stopping
criterion~(\hyperref[eq:st_crit_algorithm]{Alg2}). Indeed, we recall nested 
iteration \(u_\ell^{0, 0} = u_\ell^{0, \jj} = u_{\ell-1}^{\kk, \jj}\) for all
\(\ell \ge 1\) so that all \(u_\ell^{k, j}\) besides \(u_{\ell}^{0, 0}\) are
controlled by the subsequent \textsl{a~posteriori} 
estimate~\eqref{eq:aposteriori-jk}. 
The proof follows along the lines of~\cite[Proposition~2]{hpsv2021} 
with the modification that the linearization contracts in 
energy~\eqref{eq:energy-contraction} (instead of norm) together with the 
equivalence \eqref{eq:energy}. Details are found in the Appendix~\ref{appendix}.
\begin{proposition}[\textsl{a~posteriori} error control]
	\label{prop:aposteriori-jk}
	Consider arbitrary parameters $0 < \theta \le 1$, $1 \le \Cmark$,
	$0 < \llin$, $0 < \rho < 1$, \(\alpha_{\min} > 0\), \(J_{\max} \in \N\),
	$\tau \ge 0$, and arbitrary \(u_0^{0,0} \in \XX_0\). Suppose that the
	estimator satisfies stability~\eqref{axiom:stability} and
	reliability~\eqref{axiom:reliability}. Then, for any $(\ell,k,j) \in \QQ$
	with $k \ge 1$ and $j \ge 1$, it holds that
	\begin{equation}\label{eq:aposteriori-jk}
		\enorm{u^\exact - u_\ell^{k,j}}
		\le
		\Crel' \,
		\big[
			\eta_\ell(u_\ell^{k,j})
			+ \enorm{u_\ell^{k,j} - u_\ell^{k-1,\jj}}
			+ \enorm{u_\ell^{k,j} - u_\ell^{k,j-1}}
		\big].
	\end{equation}
	The constant $\Crel' > 0$ depends only on $\Crel, \, \Cstab, \, \qalg, \,
	\alpha, \, \Ccnt, \,  \qnrgs$. \qed
\end{proposition}
%

\section{Full linear convergence}\label{sec:full-linear convergence}

We present our first main result, namely that Algorithm~\ref{algorithm}
leads to full R-linear convergence of the quasi-error consisting of 
discretization error by means of the associated estimator 
\(\eta_\ell(u_\ell^{\star})\), linearization error 
\(\enorm{u_\ell^{\star} - u_\ell^{k,\star}}\), 
and algebraic solver error \(\enorm{u_\ell^{k,\star} - u_\ell^{k, j}}\).
While this choice for the quasi-error, in line with the stopping
criterion~(\hyperref[eq:st_crit_algorithm]{Alg2}), differs from earlier 
contributions, the proof steps remain similar. Importantly, using triangle 
inequalities and reliability~\eqref{axiom:reliability}, we can estimate the
total error $\enorm{u^\star - u_\ell^{k,j}}$ by the quasi-error.
\begin{theorem}
	[full R-linear convergence of Algorithm~\ref{algorithm}]
	\label{theorem:linearconv}
	Consider arbitrary parameters $0 < \theta \le 1$, $1 \le \Cmark$,
	$0 < \llin$, $0 < \rho < 1$, \(\alpha_{\min} > 0\), \(J_{\max} \in \N\),
	$\tau \ge 0$, and arbitrary \(u_0^{0,0} \in \XX_0\).
	Suppose that the estimator
	satisfies~\eqref{axiom:stability}--\eqref{axiom:reliability}.
	Then, Algorithm~\ref{algorithm} guarantees full R-linear convergence
	of the quasi-error
	\begin{equation}\label{eq:quasi-error}
	\Eta_\ell^{k,j}
		\coloneqq
		\eta_\ell(u_\ell^{\star})
		+ \enorm{u_\ell^\star - u_\ell^{k,\star}}
		+ \enorm{u_\ell^{k,\star} - u_\ell^{k,j}},
	\end{equation}
	i.e., there exist constants $0 <\qlin< 1$ and $\Clin > 0$ such that
	\begin{equation}\label{eq:Rlin:convergence}
		\Eta_\ell^{k,j}
		\le
		\Clin
		{\qlin}^{|\ell,k,j| - |\ell'\!,k'\!,j'|} \, \!
		\Eta_{\ell'}^{k',j'}
		\text{ for all } 
		(\ell'\!,k'\!,j'),(\ell,k,j) \in \QQ
		\text{ with }
		|\ell'\!,k'\!,j'| \! < |\ell,k,j|, \!
	\end{equation}
	where $\Clin$ and $\qlin$ depend only on $L, \, \alpha, \, \qred, \,
	\qnrg,$ $\Cstab,
	\, \Crel, \, \theta, \, \lambda_{\rm lin}$, \(j_0\), and
	\(\qnrgs\).
\end{theorem}
A core result needed in proving the above main result consists in establishing
that Algorithm~\ref{algorithm} provides energy contraction at \emph{each}
inexact linearization step, which is closely related
to~\cite[Theorem~2.1]{hpw2021}.
\begin{proposition}[energy-contraction of inexact linearization]
	\label{prop:contrlin}
	Consider arbitrary parameters $0 < \theta \le 1$, $1 \le \Cmark$,
	$0 < \llin$, $0 < \rho < 1$, \(\alpha_{\min} > 0\), \(J_{\max} \in \N\),
	$\tau \ge 0$, and arbitrary \(u_0^{0,0} \in \XX_0\).
	Then, it holds that
	\begin{equation}\label{prop:linearization:eq2}
		0
		\le
		\dist(u_\ell^\star, u_\ell^{k,\jj})
		\le
		\qnrg \, \dist(u_\ell^\star, u_\ell^{k-1,\jj})
		\quad \text{for all }
		(\ell,k,\jj) \in \QQ
	\end{equation}
	and consequently
	\begin{equation}
		\label{prop:linearization:eq2+}
		\frac{1-\qnrg}{\qnrg} \, \dist(u_\ell^\exact, u_\ell^{k,\jj})
		\le
		\dist(u_\ell^{k,\jj}, u_\ell^{k-1,\jj})
		\le
		\dist(u_\ell^\exact, u_\ell^{k-1,\jj})
		\quad \text{for all }
		(\ell,k,\jj) \in \QQ,
	\end{equation}
	where
	$
		\displaystyle 0
		\le
		\qnrg \le 1 - \frac{2 \CCeng}{L} \frac{\alpha^2}{\Ccont^2} \,
		(1-\qalg)^2
		< 1
	$
	with $0< \CCeng < \Ceng$ from Proposition~\ref{prop:unif-steps}.
\end{proposition}
\begin{proof}
	Nested iteration $u_\ell^{k,0} = u_\ell^{k-1, \jj}$, algebraic 
	contraction~\eqref{eq:alg:sol:contraction}, and $\jj[\ell,k] \ge 1$ yields
	\begin{equation*}
		\enorm{u_\ell^{k,\star} - u_\ell^{k,\jj}}
		\eqreff{eq:alg:sol:contraction}
		\le
		\qalg^{\jj[\ell,k]} \, \enorm{u_\ell^{k,\star} - u_\ell^{k,0}}
		\le
		\qalg \, \enorm{u_\ell^{k,\star} - u_\ell^{k,0}}
		=
		\qalg \, \enorm{u_\ell^{k,\star} - u_\ell^{k-1,\jj}}
	\end{equation*}
	and hence, together with a triangle inequality and 
	Proposition~\ref{prop:unif-steps}, this leads to
	\begin{equation}\label{eq1:linearization}
		\CCeng (1-\qalg)^2 \, \enorm{u_\ell^{k,\star} - u_\ell^{k-1,\jj}}^2
		\le
		\CCeng \, \enorm{u_\ell^{k,\jj} - u_\ell^{k-1,\jj}}^2
		\eqreff{eq:norm_energy_estimate_finalj}
		\le
		\dist(u_\ell^{k,\jj}, u_\ell^{k-1,\jj}).
	\end{equation}
	Next, note that
	\begin{align*}
		&\alpha \, \enorm{u_\ell^\star - u_\ell^{k-1,\jj}}^2
	\stackrel{\eqref{def:assumptions_operator},\eqref{eq:exact_solution_disc}}
		\le
		F(u_\ell^\star - u_\ell^{k-1,\jj})
		- \dual{
			\AA u_\ell^{k-1,\jj}
			}{
			u_\ell^\star - u_\ell^{k-1,\jj}
			}_{\XX' \times \XX}
		\\& \quad
		\eqreff*{eq:linearization:exact}
		=
		a(
			u_\ell^{k-1,\jj}; u_\ell^{k,\star} -u_\ell^{k-1,\jj},
			u_\ell^\star - u_\ell^{k-1,\jj}
		)
		\eqreff{eq:linearization:bounds}
		\le
		\Ccont \, \enorm{u_\ell^{k,\star} - u_\ell^{k-1,\jj}} \,
		\enorm{u_\ell^\star - u_\ell^{k-1,\jj}}.
	\end{align*}
	Together with~\eqref{eq1:linearization}, we obtain 
	\begin{equation}\label{eq2:linearization}
		0
		\le
		C \, \enorm{u_\ell^\star - u_\ell^{k-1,\jj}}^2
		\le
		\dist(u_\ell^{k,\jj}, u_\ell^{k-1,\jj})
		\quad \text{with }
		C \coloneqq \frac{\alpha^2 (1-\qalg)^2}{\Ccont^2} \, \CCeng
		> 0.
	\end{equation}
	Following the idea of \cite[Theorem~2.1]{hpw2021}, we show 
	\eqref{prop:linearization:eq2} with $0 < \qnrg \le 1 - 2C / L$
		\begin{align*}
			0
			\le \dist(u_\ell^\star, u_\ell^{k,\jj})
			&\eqreff*{eq:pythagoras_discrete}
			=
			\dist(u_\ell^\star, u_\ell^{k-1,\jj})
			- \dist(u_\ell^{k,\jj}, u_\ell^{k-1,\jj})
			\stackrel{\eqref{eq2:linearization}, \eqref{eq:energy}}
			\le
			[1 - 2C / L] \, \dist(u_\ell^\star, u_\ell^{k-1,\jj}).
		\end{align*}
		The proof of~\eqref{prop:linearization:eq2+} follows by
		using \eqref{prop:linearization:eq2}, \eqref{eq:pythagoras_discrete}, 
		and \eqref{eq:minimization}. This concludes the proof.
\end{proof}
\begin{proof}[\textbf{Proof of Theorem~\ref{theorem:linearconv}}]The proof 
follows along~\cite[Proof of Theorem~4]{bfmps2025}. Following, 
e.g.,~\cite{axioms} and \cite[Lemma~2]{bfmps2025}, there holds equivalence of 
tail summability and linear convergence of the quasi-error: For all 
$(\ell'\!,k'\!,j'),(\ell,k,j) \in \QQ$ with $|\ell'\!,k'\!,j'| \! > |\ell,k,j|$, 
there holds 
	\begin{equation*}
		\sum_{
			\substack{(\ell',k',j') \in \QQ
				\\
				|\ell',k',j'| > |\ell,k,j|
			}
		}
		\Eta_{\ell'}^{k',j'}
		\le C_{\textup{sum}}
		\Eta_\ell^{k,j} 
		\quad
		\Longleftrightarrow 
		\quad
		\Eta_{\ell'}^{k',j'}
		\le
		\Clin
		{\qlin}^{|\ell'\!,k'\!,j'|-|\ell,k,j|} \, \!
		\Eta_\ell^{k,j},
	\end{equation*}
	where \(\Clin = 1 + C_{\textup{sum}}\) and \(q_{\textup{lin}} = C_{\textup{sum}} / (1 + C_{\textup{sum}})\).
	Thus, we proceed to show tail summability of the quasi-error.

	\textbf{
		Step~1
		(Quasi-monotonicity of $\boldsymbol{\Eta_\ell^{k,j}}$ in ${\bm j}$).
		}
	For all $0 \le j \le j' \le \jj [\ell, k]$, solver
	contraction~\eqref{eq:alg:sol:contraction} yields
	\begin{align}
		\Eta_\ell^{k,j'}
		&=
		\eta_\ell(u_\ell^{\star})
		+ \enorm{u_\ell^\star - u_\ell^{k,\star}}
		+ \enorm{u_\ell^{k,\star} - u_\ell^{k,j'}}
		\nonumber
		\\
		&\eqreff*{eq:alg:sol:contraction}
		\le
		\eta_\ell(u_\ell^{\star})
		+ \enorm{u_\ell^\star - u_\ell^{k,\star}}
		+ \qalg^{j-j'} \enorm{u_\ell^{k,\star} - u_\ell^{k,j}}
		\le  \Eta_\ell^{k,j}. \label{eq:algHlkj}
	\end{align}
	Since Proposition~\ref{prop:unif-steps} ensures that the total number of
	algebraic solver steps is uniformly bounded, i.e.,
	$\jj [\ell, k] \le j_0 < \infty $ for all $(\ell, k,0) \in \QQ$, the
	previous estimate proves
	\begin{equation}\label{eq:quasimonEta}
		\medmuskip = -1mu
		\sum_{j'= j}^{\jj[\ell, k]} \Eta_\ell^{k,j'}
		\,
		\eqreff*{eq:algHlkj}
		\le \,
		\bigl(\jj[\ell, k] +1 \bigr)  \, \Eta_\ell^{k,j}
		\le
		\bigl(j_0 + 1 \bigr)  \, \Eta_\ell^{k, j}
		\
		\text{for all \((\ell, k, j) \in \QQ\)}.
	\end{equation}

	\textbf{Step~2 (stability of energy-based quasi-error).}
	Consider an energy-based
	quasi-error for the final algebraic solver step
	%
	$	\Eta_\ell^k
		\coloneqq
		\dist( u_\ell^{\star},u_\ell^{k,\jj})^{1/2}
		+ \eta_\ell(u_\ell^{k,\jj}) 
		\text{ for all} \
		(\ell,k,\jj) \in \QQ. $
	%
	First, using stability~\eqref{axiom:stability} and a 
	triangle inequality, one obtains 
	for all $(\ell,k,\jj) \in \QQ$ 
	\begin{equation}\label{eq:Eta_nrg_UB}
		\Eta_\ell^{k}
		\eqreff{eq:energy}
		\le
		(L/2)^{1/2}  \, \enorm{u_\ell^\star - u_\ell^{k,\jj}}
		+ \eta_\ell(u_\ell^{k,\jj})
		\lesssim
		\Eta_\ell^{k, \jj}.
	\end{equation}
	Next, recall that
	$u_\ell^{0,0} = u_\ell^{0,\jj} = u_\ell^{0,\star}$ by definition and 
	use of stability~\eqref{axiom:stability} leads to
	\begin{equation*}
		\Eta_{\ell}^{0, 0}
		\eqreff{axiom:stability}
		\simeq
		\enorm{u_\ell^\star - u_\ell^{0, 0}} + \eta_\ell(u_\ell^{0, 0})
		\eqreff{eq:energy}
		\le
		(2/\alpha)^{1/2} \, \dist(u_\ell^\star, u_\ell^{0, 0})^{1/2}
		+ \eta_\ell(u_\ell^{0, 0})
		\simeq
		\Eta_{\ell}^0.
	\end{equation*}
	For $k \ge 1$, using $u_\ell^{k,0} = u_\ell^{k-1,\jj}$, 
	stability~\eqref{axiom:stability},
	equivalence~\eqref{eq:energy}, and contraction~\eqref{eq:energy-contraction}
	imply
	\begin{align*}
		\Eta_{\ell}^{k, 0} \,
		&\eqreff*{axiom:stability}\le
		2 \, \enorm{u_\ell^\star - u_\ell^{k,\star}}
		+ (\Cstab +1) \, \enorm{u_\ell^\star - u_\ell^{k-1,\jj}}
		+ \eta_{\ell}(u_\ell^{k-1, \jj})
		\\
		&\eqreff*{eq:energy}
		\le
		2 \,  (\Cstab +1) \, (2/\alpha)^{1/2} \,
		\bigl[
			\dist(u_{\ell}^{\star},u_{\ell}^{k,\star}) +
			\dist(u_{\ell}^{\star},u_{\ell}^{k-1,\jj})
		\bigr]^{1/2}
		+ \eta_{\ell}(u_\ell^{k-1, \jj})
		\eqreff*{eq:energy-contraction}
		\lesssim \,
		\Eta_\ell^{k-1}.
	\end{align*}
	Therefore, we derive that
	\begin{equation}\label{eq1:step9}
		\Eta_\ell^{k,0}
		\lesssim \,
		\Eta_\ell^{(k-1)_+}
		\quad \text{for all }
		(\ell,k,0) \in \QQ,
		\quad \text{where }
		(k-1)_+ \coloneqq \max\{0, k-1 \}.
	\end{equation}

	\textbf{
		Step~3.
		(tail-summability of $\boldsymbol{\Eta_\ell^{k}}$ with respect to
		$\boldsymbol{\ell}$ and $\boldsymbol{k}$).
		}
	Thanks to~\cite[Theorem~4]{ghps2021}, there exist
	$\gamma>0$, $\widetilde{C}_{\textup{lin}} > 0$, and
	$0 < \widetilde q_{\rm lin} < 1$ depending only on
	$L, \, \alpha, \, \qred, \, \qnrg,$ $\Cstab, \, \Crel, \, \theta, \,
	\lambda_{\rm lin}$ such that the weighted quasi-error quantity
	\begin{subequations}\label{eq:lambda-sumlk}
		\begin{equation}
			\Lambda_\ell^{k}
			\coloneqq
			\bigl[
				\dist(u^{\star},u_{\ell}^{k,\jj})
				+ \gamma \eta_{\ell}(u_{\ell}^{k,\jj})^{2}
			\bigr]^{1/2}
		\end{equation}
		satisfies
		\begin{equation}
			\Lambda_{\ell'}^{k'}
			\le
			\widetilde{C}_{\rm lin}
			\widetilde{q}_{\rm lin}^{|\ell',k'|-|\ell,k|} \,
			\Lambda_{\ell}^{k}
			\quad \text{for all} \
			(\ell,k,\jj), \,
			(\ell',k',\jj') \in \QQ,
			\text{ with }
			|\ell',k'|>|\ell,k|
		\end{equation}
	\end{subequations}
	with respect to the reduced counter
	\(
		| \ell', k' |
		\coloneqq
		\# \{(\ell, k, \jj) \colon |\ell, k, \jj | < |\ell', k', \jj |\}
	\).
	Note that
	\begin{equation*}
		\dist(u^{\star}_\ell,u_{\ell}^{k,\jj})
		\eqreff{eq:pythagoras_discrete}=   \dist(u^{\star},u_{\ell}^{k,\jj}) -
		\dist(u^{\star},u_{\ell}^\star)
		\eqreff{eq:minimization}\le  \dist(u^{\star},u_{\ell}^{k,\jj})
	\end{equation*}
	and
	\begin{align*}
		\dist(u^{\star},u_{\ell}^{k,\jj})
		&\eqreff*{eq:pythagoras_discrete}
		=
		\dist(u^{\star},u_{\ell}^\star)
		+ \dist(u^{\star}_\ell,u_{\ell}^{k,\jj})
		\eqreff*{eq:energy}
		\lesssim
		\enorm{u^\star - u_\ell^\star }^2
		+ \dist(u^{\star}_\ell,u_{\ell}^{k,\jj})
		\\&
		\eqreff*{axiom:reliability}
		\lesssim
		\eta_\ell(u_\ell^\star)^2
		+ \dist(u^{\star}_\ell,u_{\ell}^{k,\jj})
		\eqreff{axiom:stability}
		\lesssim
		\eta_\ell(u_\ell^{k,\jj})^2
		+ \enorm{u^{\star}_\ell - u_{\ell}^{k,\jj}}^2
		+ \dist(u^{\star}_\ell,u_{\ell}^{k,\jj})
		\\&
		\eqreff*{eq:energy}
		\lesssim
		\eta_\ell(u_\ell^{k,\jj})^2 + \dist(u^{\star}_\ell,u_{\ell}^{k,\jj}).
	\end{align*}
	Since $\gamma$ is fixed, the above two estimates lead to
	\begin{equation}\label{eq:etalam}
		\Eta_\ell^{k}
		\simeq
		\Lambda_\ell^{k}
		\quad \text{for all} \
		(\ell,k,\jj) \in \QQ.
	\end{equation}
	With the geometric series, this yields
	\begin{equation}\label{eq:step8}
		\sum_{
			\substack{
				(\ell',k',\jj) \in \QQ
				\\
				|\ell',k',\jj| > |\ell,k,\jj|
			}
		}
		\Eta_{\ell'}^{k'}
		\eqreff{eq:etalam}
		\simeq
		\sum_{
			\substack{
				(\ell',k',\jj) \in \QQ
				\\
				|\ell',k',\jj| > |\ell,k,\jj|
			}
		}
		\Lambda_{\ell'}^{k'}
		\eqreff{eq:lambda-sumlk}
		\lesssim
		\Lambda_\ell^k
		\eqreff{eq:etalam}
		\simeq
		\Eta_\ell^k
		\quad \text{for all }
		(\ell,k,\jj) \in \QQ.
	\end{equation}

	\textbf{
		Step~4 (tail-summability of $\boldsymbol{\Eta_\ell^{k,j}}$ with
		respect to $\boldsymbol{\ell}$, $\boldsymbol{k}$, and
		$\boldsymbol{j}$).
		}
	Finally, for $(\ell,k,j) \in \QQ$, we observe that
	\begin{align*}
		&\sum_{
			\substack{(\ell',k',j') \in \QQ
				\\
				|\ell',k',j'| > |\ell,k,j|
			}
		}
		\Eta_{\ell'}^{k',j'}
		=
		\sum_{j'=j+1}^{\jj[\ell,k]} \Eta_{\ell}^{k,j'}
		+ \sum_{k'=k+1}^{\kk[\ell]} \sum_{j'=0}^{\jj[\ell,k']} \Eta_\ell^{k',j'}
		+ \sum_{\ell' = \ell+1}^\elll \sum_{k'=0}^{\kk[\ell']}
			\sum_{j'=0}^{\jj[\ell',k']} \Eta_{\ell'}^{k',j'}
		\\& \qquad
		\eqreff{eq:quasimonEta}
		\lesssim
		\Eta_{\ell}^{k,j}
		+ \sum_{k'=k+1}^{\kk[\ell]} \Eta_\ell^{k',0}
		+ \sum_{\ell' = \ell+1}^\elll \sum_{k'=0}^{\kk[\ell]}
			\Eta_{\ell'}^{k',0}
		\\& \qquad
		\eqreff{eq1:step9}
		\lesssim
		\Eta_{\ell}^{k,j}
		+ \sum_{k'=k}^{\kk[\ell]-1} \Eta_\ell^{k'}
		+ \sum_{\ell' = \ell+1}^\elll \sum_{k'=0}^{\kk[\ell]-1}
			\Eta_{\ell'}^{k'}
		\\& \qquad \
		\lesssim
		\Eta_{\ell}^{k,j} + \Eta_\ell^{k}
		+ \sum_{
			\substack{
				(\ell',k',\jj) \in \QQ
				\\
				|\ell',k',\jj| > |\ell,k,\jj|
				}
			}
			\Eta_{\ell'}^{k'}
		\eqreff{eq:step8}
		\lesssim
		\Eta_\ell^{k,j} + \Eta_\ell^{k}
		\eqreff{eq:Eta_nrg_UB}
		\lesssim
		\Eta_\ell^{k,j} + \Eta_\ell^{k,\jj}
		\eqreff{eq:algHlkj}
		\lesssim
		\Eta_\ell^{k,j}.
	\end{align*}
	This tail-summability is equivalent to R-linear
	convergence~\eqref{eq:Rlin:convergence} of $\Eta_\ell^{k,j}$,
	see~\cite[Lemma~2]{bfmps2025} and thus concludes the proof.
\end{proof}
Full R-linear convergence has the following crucial consequence that
follows from the geometric series; see, e.g.~\cite{aisfem}: If the rate of
convergence $s>0$ is achievable with respect to the degrees of freedom, it is
also achievable with respect to the cumulative computational costs. Hence, full
linear convergence is the key to optimal complexity.
\begin{corollary}
	[rates = complexity {\cite[Corollary~13]{bfmps2025}}]
	\label{corollary:rates:complexity}
	Suppose full R-linear convergence~\eqref{eq:Rlin:convergence}. Then, for
	any $s > 0$, it holds that
	\begin{equation}\label{eq:equiv-cost-rates}
		\medmuskip = -4mu
		M(s)
		\hspace{-0.1cm}
		\coloneqq
		\hspace{-0.2cm}
		\sup_{(\ell,k,j) \in \QQ} (\#\TT_\ell)^s \, \Eta_\ell^{k,j}
		\le
		\hspace{-0.2cm}
		\sup_{(\ell,k,j) \in \QQ}
		\Bigl(
			\hspace{-0.2cm}
			\sum_{
				\substack{
					(\ell',k',j') \in \QQ
					\\
					|\ell',k',j'| \le |\ell,k,j|
				}
			}
			\hspace{-0.4cm}
			\#\TT_{\ell'}
		\Bigr)^s
		\Eta_\ell^{k,j}
		\le
		\frac{\Clin}{\bigl(1 \ - \ {\qlin}^{1/s}\bigr)^s} \, M(s).
	\end{equation}
	Moreover, there exists $s_0 > 0$ such that
	\begin{equation}\label{eq:corMs}
		M(s) < \infty
		\quad \text{for all }
		0 < s \le s_0.\qed
	\end{equation}
\end{corollary}
As a second corollary to full linear convergence, we can characterize the limit
of
Algorithm~\ref{algorithm} for \(\elll < \infty\) and infinitesimal tolerance
$\tau =0$ in the stopping
criterion~(\hyperref[eq:st_crit_algorithm]{Alg2}).
%
\begin{corollary}\label{corollary_0tol}
	Suppose the assumptions of Theorem~\ref{theorem:linearconv} hold and
	$\tau = 0$. If $\elll < \infty$, then
	\begin{align}
		\text{either } \quad
		&\kk [\elll] < \infty
		\text{ and }
		u_\elll^{\kk, \jj} = u_\elll^\star = u^\star
		\text{ with }
		\eta_\elll (u_\elll^{\kk, \jj}) = 0,
		\\
		\text{or } \quad
		&\kk [\elll] = \infty
		\text{ and }
		u_\elll^{k, \jj} \neq u_\elll^\star = u^\star
		\text{ for all }
		k \in \N
		\text{ with }
		\eta_\elll (u_\elll^\star) = 0.
		\label{eq:case_k_infty}
	\end{align}
\end{corollary}
\begin{proof}
	While the proof of~\eqref{eq:case_k_infty} is included 
	in~\cite[Corollary~6]{ghps2021}, we provide a proof of the first statement.
	If $\elll < \infty $ and $\kk [\elll] < \infty$, then $\# \QQ < \infty$ by
	Proposition~\ref{prop:unif-steps}. Hence, Algorithm~\ref{algorithm}
	terminates at Step~(i.b.III) with
	$
		\eta_\ell(u_\ell^{k,j})
		+ \enorm{u_\elll^{\kk,\jj} - u_\elll^{\kk-1,\jj}}
		+ \enorm{u_\elll^{\kk,\jj} - u_\elll^{\kk,\jj-1}}
		= 0
	$.
	Therefore, Proposition~\ref{prop:aposteriori-jk} yields
	$u^\star = u_\elll^{\kk,\jj}$ and the Céa lemma~\eqref{eq:cea} also yields
	$u_\elll^\star = u_\elll^{\kk,\jj}$. 
\end{proof}

\section{Quasi-optimal computational costs}\label{sec:optimality}

To describe whether the
solution $u^\star$ can be approximated at rate $s>0$, we use the notion of
nonlinear approximation classes~\cite{bdd2004, stevenson2007, ckns2008, axioms} 
\begin{equation*}
	\norm{u^\exact}_{\A_s}
	\coloneqq
	\sup_{N \in \N_0}
	\Bigl(
		\bigl( N+1 \bigr)^s
		\min_{\TT_{\rm opt} \in \T_N } \eta_{\rm opt}(u^\star_{\rm opt})
	\Bigr),
\end{equation*}
where $\eta_{\rm opt}(u^\star_{\rm opt})$ is the estimator for
the (unavailable) exact Galerkin solution \(u_{\textup{opt}}^\star\) on
an optimal $\TT_{\rm opt} \in \T_N$. The following theorem is the second main
result of this work, it states that for sufficiently small 
mesh-refinement parameter $\theta$ and linearization parameter $\llin$, the 
adaptive algorithm leads to optimal (in the sense of the nonlinear
approximation classes) decrease of the quasi-error with respect to 
overall computational cost.
The proof can be derived following steps analogous to~\cite[Proof of
Theorem~3]{aisfem}. For details, we refer to the Appendix~\ref{appendix}.
\begin{theorem}[optimal complexity]\label{th:optimal_complexity}
	Consider arbitrary parameters $0 < \theta \le 1$, $1 \le \Cmark$,
	$0 < \llin$, $0 < \rho < 1$, \(\alpha_{\min} > 0\), \(J_{\max} \in \N\),
	$\tau = 0$, and arbitrary \(u_0^{0,0} \in \XX_0\).
	Suppose that the estimator satisfies
	\eqref{axiom:stability}--\eqref{axiom:discrete_reliability}.
	Let
	\begin{align}\label{eq:lambdalin}
		\begin{split}
			\lambda_{\rm lin}^\star
			&\coloneqq
			\min
			\{1, \Cstab^{-2} \, C_{\textup{solve}}^{-2}\}
			\quad
			\text{with}
			\quad
			C_{\textup{solve}}
			\coloneqq
			\Bigl(\frac{\qnrg}{1-\qnrg}\Bigr)^{1/2}\, (2/ \alpha)^{1/2},
			\\
			\theta^\star
			&\coloneqq
			(1 + \Cstab^2 \, \Cdrel^2)^{-1}.
		\end{split}
	\end{align}
	Suppose that $\theta$, $\lambda_{\rm lin} $ are sufficiently small
	in the sense of
	\begin{equation}\label{eq:thetamark}
		0 < \lambda_{\rm lin} < \lambda_{\rm lin}^\star
		\quad
		\text{and}
		\quad
		0
		<
		\thetamark
		\coloneqq
		\frac{
			(\theta^{1/2}
			+ \, (\lambda_{\rm lin} / {\lambda_{\rm lin}^\star})^{1/2})^{2}
			}{
			(1 - (\lambda_{\rm lin} / {\lambda_{\rm lin}^\star})^{1/2})^{2}
			}
		<
		\theta^\exact
		<
		1.
	\end{equation}
	Then, Algorithm~\ref{algorithm} guarantees, for all $s > 0$, that
	\begin{equation}\label{eq:optimal_complexity}
		\copt \, \norm{u^\exact}_{\A_s}
		\le
		\sup_{(\ell,k, j) \in \QQ}
		\Bigl(
			\sum_{
				\substack{
					(\ell', k', j') \in \QQ
					\\
					|\ell', k', j'| \le |\ell, k, j|
					}
				}
				\#\TT_{\ell'}
		\Bigr)^{s} \,
		\Eta_\ell^{k,j}
		\le
		\Copt \, \max\{ \norm{u^\exact}_{\A_s}, \, \Eta_{0}^{0,0}\}.
	\end{equation}
	The constant \(\copt > 0\) depends only on \(\Cstab\), the use of
	NVB refinement in \(\R^d\), and \(s\); while the constant $\Copt$ depends
	only on $\Cstab$, $\Cdrel$, $\Cmark$, $\Cmesh$, \(\alpha\), \(\qalg\),
	\(C_{\textup{nrg}}^\prime\), \(\lambda_{\textup{lin}}\), \(\qnrg\),
	$\Clin$, $\qlin$, $\# \TT_{0}$, and $s$.
	In particular, there holds optimal complexity of Algorithm~\ref{algorithm}.
\end{theorem}
The proof of Theorem~\ref{th:optimal_complexity} 
employs the following result whose proof follows the Steps~3--4 
in~\cite[Proof of Theorem~7]{ghps2021} with
\(u_\ell^{\kk,\jj}\) replacing \(u_\ell^{\kk}\), we refer to the 
Appendix~\ref{appendix} for details.
\begin{lemma}[estimator equivalence]
	\label{lem:estimator_equivalence}
	Suppose the assumptions of Theorem~\ref{th:optimal_complexity} hold.
	Recall \(\lambda_{\rm lin}^\star\) from
	Theorem~\ref{th:optimal_complexity}. Then, for all
	\(0 < \theta \le 1\), all
	\(
		0
		<
		\lambda_{\rm lin}
		<
		\lambda_{\rm lin}^\star
	\),
	and \(\thetamark\) defined in~\eqref{eq:thetamark}, there holds the
	equivalence
	\begin{equation}\label{eq:estimator_equivalence}
		\bigl[
			1-(\lambda_{\rm lin}/\lambda_{\rm lin}^\star)^{1/2}
		\bigr] \,
		\eta_{\ell}(u_\ell^{\kk, \jj})
		\le
		\eta_{\ell}(u_\ell^\star)
		\le
		\bigl[
			1 + (\lambda_{\rm lin}/\lambda_{\rm lin}^\star)^{1/2}
		\bigr] \,
		\eta_{\ell}(u_\ell^{\kk, \jj}).
	\end{equation}
	In particular, we have that Dörfler marking for
	(\(u_\ell^\star\), \(\thetamark\)) implies Dörfler marking for
	(\(u_\ell^{\kk, \jj}\), \(\theta\)), i.e., for any
	\(\RRR_{\ell} \subseteq \TT_{\ell}\), there holds the following implication:
	\begin{equation}\label{eq:equivalence_Doerfler}
		\thetamark \, \eta_{\ell}(u_\ell^\star)^2
		\le
		\eta_{\ell}(\RRR_{\ell}; u_\ell^\star)^2
		\quad
		\Longrightarrow
		\quad
		\theta \, \eta_{\ell}(u_\ell^{\kk, \jj})^2
		\le
		\eta_{\ell}(\RRR_{\ell}; u_\ell^{\kk, \jj})^2.
	\end{equation}
\end{lemma}

\section{Numerical experiments}\label{sec:numerics}

In the following experiments, we highlight the
parameter-free Kačanov linearization (labelled K in the plots) combined with an 
optimal $hp$-robust multigrid (MG) method from~\cite{imps2022}. For comparison, 
we present results for the standard Zarantonello iteration (Z), 
and a Newton method (N) with different optimal algebraic solvers. 
Experiments are conducted in the open source Matlab package
MooAFEM~\cite{MooAFEM}.

\subsection{Computational examples}
We consider the following test cases.

\medskip
\subsubsection{L-shaped domain.}\label{section:HPW_benchmark}
From~{\cite[Section~5.3]{hpw2021}}, we consider the L-shaped domain
\(\Omega =
(-1, 1)^2 \setminus [0, 1] \times [-1, 0]\), the nonlinear diffusion
coefficient \(\mu(t) = 1 + \exp(-t)\), and right-hand side \(f\) such that the
exact solution, given in polar coordinates, reads
\begin{equation*}
	u^\star(r, \varphi)
	\coloneqq
	r^{2/3} \, \cos(2 \varphi / 3) (1-r \cos(\varphi)) (1 + r \cos(\varphi))
	(1- r \sin(\varphi)) (1  + r \sin(\varphi)) \cos(\varphi).
\end{equation*}
Then, the growth assumption~\eqref{eq:assumption-mu} is satisfied with
\(\alpha = 1 - 2 \exp(-3/2)\) and \(L = 6\). This leads to $\delta = 1/L = 1/6$
as the optimal damping (with respect to energy) parameter in the case of the
Zarantonello iteration employed in the numerical test. The obtained adaptive
mesh for this problem is presented in Figure~\ref{fig:mesh} (left).

\medskip
\subsubsection{Z-shaped domain.}\label{section:GHPS_benchmark}
From~{\cite[Section~6.2]{ghps2021}, we consider the Z-shaped domain
$\Omega = (-1, 1)^2 \setminus \mathrm{conv}\{(-1, 0), (0, 0), (-1, -1)\}$,
the nonlinear diffusion coefficient $\mu(t) = 1 + \log(1+t)/(1+t)$, and
right-hand side $f=1$ and $\boldsymbol{f} =0$. Then, the growth
assumption~\eqref{eq:assumption-mu} is satisfied with
$\alpha \approx 0.9582898017$ and $L \approx 1.542343818$. Thus, we use
$\delta = 1/L \approx 0.648364$ as the optimal damping parameter in the case of
the	Zarantonello iteration. The obtained adaptive mesh for this problem is
presented in Figure~\ref{fig:mesh} (right).
\begin{figure}
	\resizebox{\textwidth}{!}{
		\subfloat{
		\includegraphics{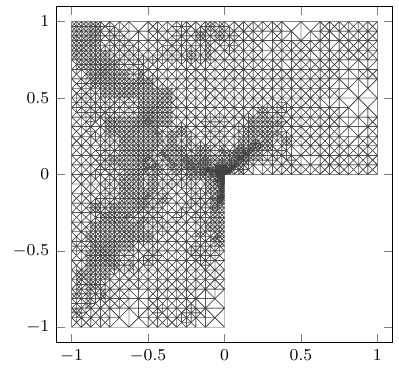}
		}
		\hfil
		\subfloat{
		\includegraphics{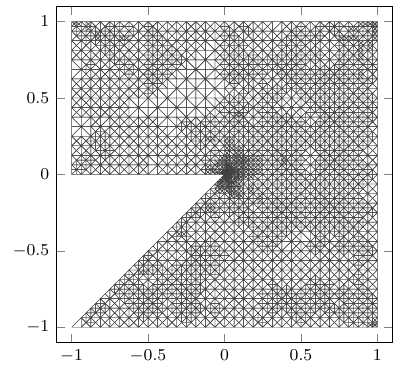}
		}
	}
	\caption{\label{fig:mesh} Adaptive mesh for the L-shaped
	domain problem from Section~\ref{section:HPW_benchmark} with 
	\(\# \TT_{7} = 6484\) (left) and for the Z-shaped
	domain problem from Section~\ref{section:GHPS_benchmark} with 
	\(\# \TT_{6} = 6122\) (right) for adaptivity
	parameters \(\theta = 0.5\) and \(\lambda_{\rm lin} = 0.7\).}
\end{figure}
\subsection{Optimality of Algorithm~\ref{algorithm}}
From Theorem~\ref{th:optimal_complexity}, we know that optimal complexity is
guaranteed for sufficiently small adaptivity parameters $\lambda_{\rm lin}$ for
the linearization stopping criterion~(\hyperref[eq:st_lin]{Alg4}) and $\theta$ for
marking~\eqref{eq:doerfler}, while the choice of \(\Cmark \ge 1\) and
\(0 < \rho < 1\), \(\alpha_{\min} > 0\), \(J_{\max} \in \N\), and
\(u_0^{0, 0} \coloneqq 0\) is arbitrary. Throughout the upcoming numerical
tests, we use \(\Cmark =1\) and $\rho = 1/2$ in Algorithm~\ref{algorithm}(i.d),
\(\alpha_{\min} = 100\) and \(J_{\max} = 1\) in~(\hyperref[eq:st_alg]{Alg3}), and
\(u_0^{0, 0} \coloneqq 0\). Table~\ref{tab:cost} presents a thorough numerical
study of the weighted cumulative computational time for different choices of
adaptivity parameters. It is interesting to note that the performance of the
algorithm is, in practice, not restricted to small parameters. On the contrary,
a favorable combination of parameters seems to be $\lambda_{\rm lin} = 0.9$ and
$\theta = 0.5$ which is used henceforth in the remaining numerical experiments.
Optimality of the algorithm is showcased both with respect to the degrees of
freedom and with respect to the cumulative computation time in
Figure~\ref{fig:optimality_HPW} for the L-shaped problem, and
Figure~\ref{fig:optimality_GHPS} for the Z-shaped problem, respectively. In
particular, we highlight the performance of the parameter-free Ka{\v c}anov
linearization, although the experiments employ the optimal damping
parameter for the Zarantonello iteration, and a damped 
Newton iteration. Moreover, we emphasize that optimal complexity 
is also observed heuristically for higher-order discretizations, see 
Figure~\ref{fig:HO}, though this is not 
yet covered; see Proposition~\ref{prop:axioms} and the comments from 
Section~\ref{sec:comments}.
\begin{table}
	\resizebox{\textwidth}{!}{
		\includegraphics{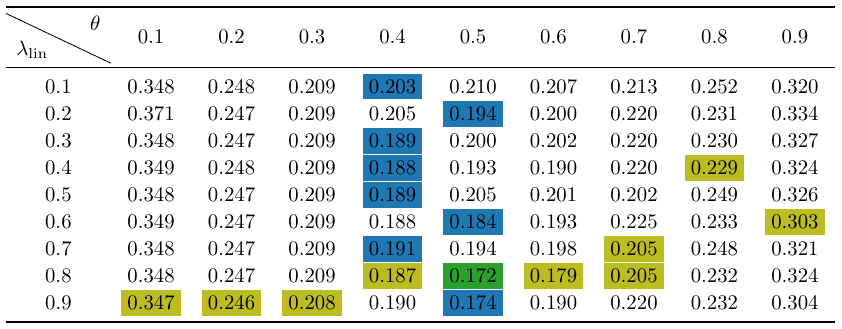}
	}
	\caption{\label{tab:cost} Optimal selection of parameters 
	for the Ka{\v c}anov linearization 
	and L-shaped domain problem considering the
	estimator-weighted cumulative time
	[$\eta_\ell(u_\ell^{\kk, \jj})
			\bigl(
			\sum_{
				|\ell',k',j'| < |\ell,\kk,\jj|
				}
			\mathtt{time}(\ell')
			\bigr)^{1/2}$
	]
	and
	stopping criterion $\eta_\ell(u_\ell^{\kk, \jj}) \ < 10^{-2}$
	for various $\theta$ and
	$\lambda_{\rm lin}$. Blue: the rowwise best combination; 
	yellow: the
	columnwise best combination; green: if both coincide.}  
\end{table}
\begin{figure}
		\resizebox{\textwidth}{!}{
		\subfloat{\includegraphics[valign=t]{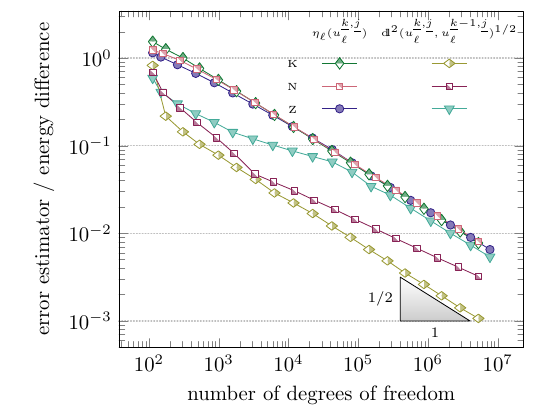}}
		\hfil
		\subfloat{\includegraphics[valign=t]{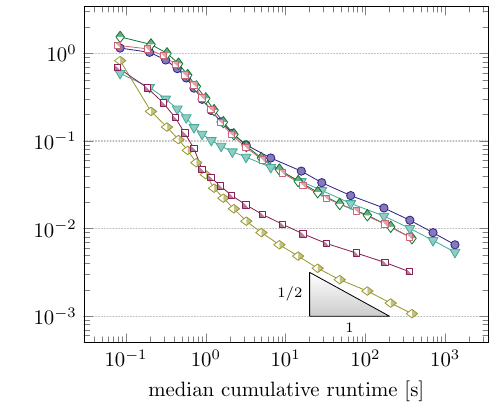}}
		}
	\caption{\label{fig:optimality_HPW} 
	Convergence of the
	error estimator \(\eta_\ell(u_\ell^{\kk, \jj})\) and
	energy difference~\({\rm d\!l}^2(u_\ell^{\underline{k}, \underline{j}},
	u_\ell^{\underline{k}-1, \underline{j}})^{1/2}\) for the L-shaped domain 
	problem from Section~\ref{section:HPW_benchmark}
	with parameters \(\theta = 0.5\) and \(\lambda_{\rm lin} = 0.7\) as well as 
	\(\delta = 1/6\) (Zarantonello iteration) resp. \(\delta = 0.5\) 
	(Newton iteration).}
\end{figure}
\begin{figure}
	\resizebox{\textwidth}{!}{
		\subfloat{
			\includegraphics[valign=t]{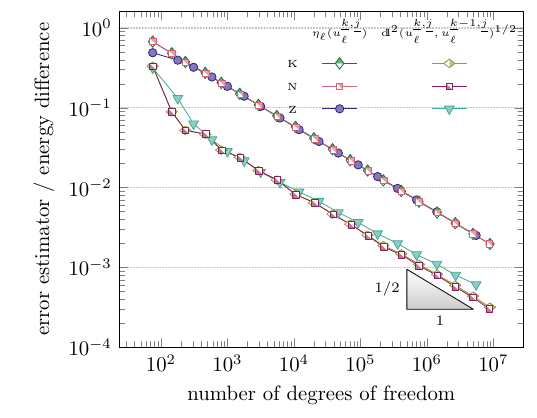}
		}
		\hfil
		\subfloat{
			\includegraphics[valign=t]{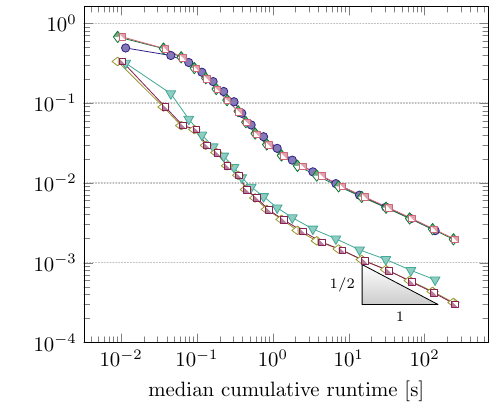}
		}
	}
	\caption{\label{fig:optimality_GHPS} 
	Convergence of the
	error estimator \(\eta_\ell(u_\ell^{\kk, \jj})\) and energy difference
	\({\rm d\!l}^2(u_\ell^{\underline{k}, \underline{j}},
	u_\ell^{\underline{k}-1, \underline{j}})^{1/2}\) for the Z-shaped domain 
	problem from Section~\ref{section:GHPS_benchmark}
	with parameters \(\theta = 0.5\) and \(\lambda_{\rm lin} = 0.7\) as well as 
	\(\delta \approx 0.648364\) (Zarantonello iteration) resp. \(\delta = 0.5\) 
	(Newton iteration).}
\end{figure}

\begin{figure}
	\resizebox{\textwidth}{!}{
		\subfloat{
			\includegraphics[valign=t]{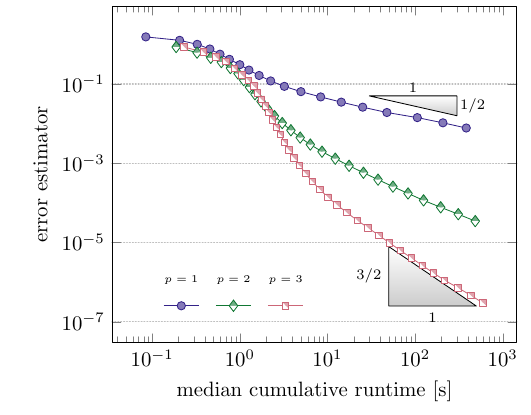}
		}
		\hfil
		\subfloat{
			\includegraphics[valign=t]{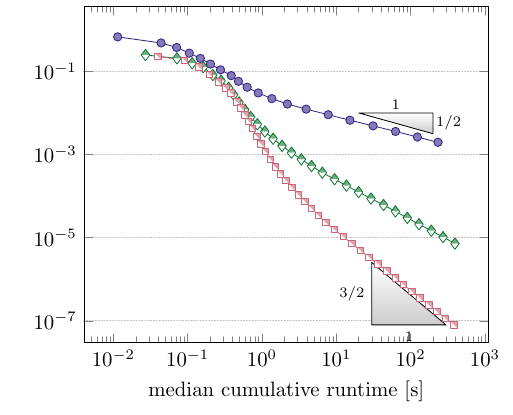}
		}
	}
	\caption{\label{fig:HO} Convergence 
	of the estimator \(\eta_\ell(u_\ell^{\kk, \jj})\) and energy difference
			\({\rm d\!l}^2(u_\ell^{\underline{k}, \underline{j}},
			u_\ell^{\underline{k}-1, \underline{j}})^{1/2}\) using the Kačanov 
	linearization for the L-shaped domain problem from 
	Section~\ref{section:HPW_benchmark} (left) and Z-shaped domain problem 
	from Section~\ref{section:GHPS_benchmark} (right) for several polynomial 
	degrees \(p=1,2,3\).}
\end{figure}

\begin{figure}
	\resizebox{\textwidth}{!}{
		\subfloat{
			\includegraphics{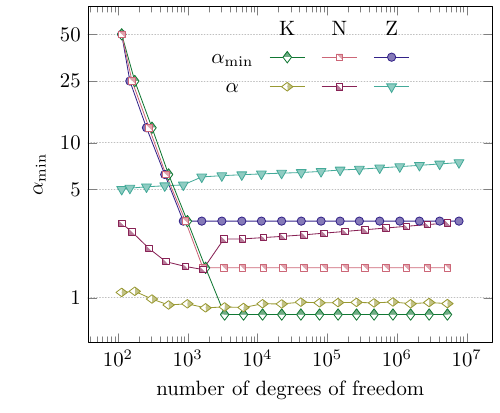}
		}
		\hfil
		\subfloat{
			\includegraphics{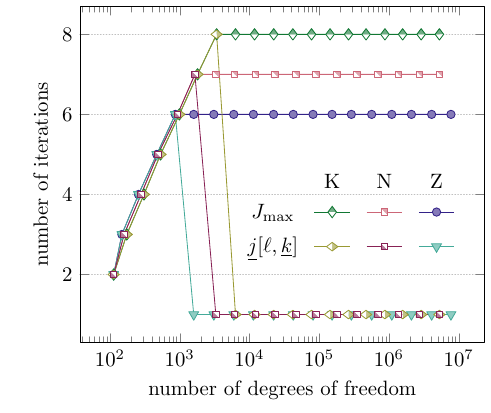}
		}
	}
	\caption{\label{fig:iterations} 
	History plot of $\alpha_\ell^{\kk, \jj}$ and
		$\alpha_{\min}$ (left), and iteration counters
		$\underline{j}[\ell, \kk]$ and $J_{\max}$ (right) for the L-shaped 
		domain problem from Section~\ref{section:HPW_benchmark} with parameters 
		$\theta = 0.5$ and $\lambda_{\rm lin} = 0.7$.}
\end{figure}

\subsection{Performance of the novel stopping criterion~{\bf 
(\hyperref[eq:st_alg]{Alg3})}}
We now investigate numerically the performance of the novel parameter-free
stopping criterion for the algebraic solver. In Figure~\ref{fig:iterations}
(left), the running parameter $\alpha_\ell^{\kk,\jj}$ is displayed together
with the evolution of the lower bound $\alpha_{\rm min}$. Though these values
are typically small for the Ka{\v c}anov iteration, one can see in
Figure~\ref{fig:iterations} (right) that the number of algebraic solver steps
$\jj[\ell,\kk]$ is rather not affected. Indeed, after the
pre-asymptotic regime, only one step is needed for both linearization methods.
In particular, the bounds $J_{\max}$ on the number of iterations of the
algebraic solver are commensurate.
%

\subsection{Alternative choices of iteration schemes}
In this subsection, we investigate the performance of the adaptive algorithm for 
a Newton method for several damping parameters and an alternative selection of 
an algebraic solver.
To this end, we consider the Z-shaped domain problem from 
Section~\ref{section:GHPS_benchmark} and a Newton method combined with
an optimally preconditioned conjugate gradient method (PCG) in the spirit 
of~\cite{cnx2012,hpsv2021}. For this \emph{particular} problem, the 
theoretical bounds from Section~\ref{sec:damped_newton} also allow for an 
undamped Newton method. The numerical experiments confirm the optimal complexity 
for both considered uniformly contractive algebraic solvers, see 
Figure~\ref{fig:Newton-alg}. Moreover, we observe that the \emph{total} 
number of iterations
is not significantly affected by the choice of
the damping parameter, see
Figure~\ref{fig:Newton-damping}: a larger damping parameter \(\delta\) 
leads to only one Newton step
at the expense of a higher number of algebraic solver iterations.  

}

\begin{figure}
	\includegraphics{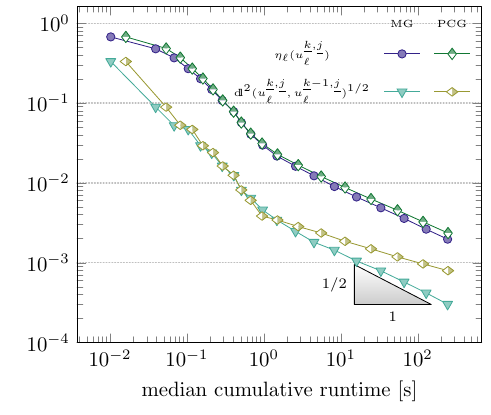}
	\caption{\label{fig:Newton-alg}
	Optimal complexity
	of an undamped Newton method with an optimally preconditioned conjugate 
	gradient method in the spirit of~\cite{cnx2012, hpsv2021} and MG solver from
	\cite{imps2022} for the Z-shaped domain problem from 
	Section~\ref{section:GHPS_benchmark} with parameters $\theta = 0.5$ and 
	$\lambda_{\rm lin} = 0.7$.}
\end{figure}

\begin{figure}
	\resizebox{\textwidth}{!}{
		\subfloat{
			\includegraphics{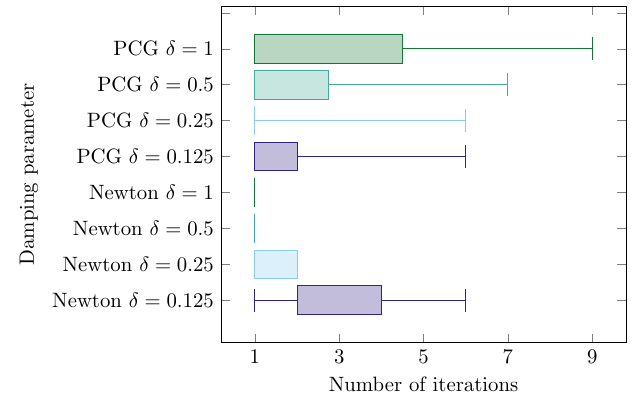} 
		}
		\hfil
		\subfloat{
			\includegraphics{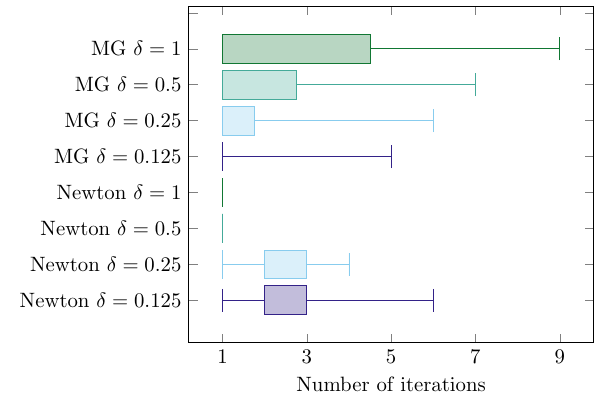} 
		}
	}
	\caption{\label{fig:Newton-damping} Maximum, median, and 
	minimum number of iterations for several damping parameters 
	\(\delta \in  \{1/8, 1/4, 1/2, 1 \}\) in the Newton 
	method with an optimally preconditioned conjugate gradient method (left) and 
	MG solver (right) for the Z-shaped domain problem from 
	Section~\ref{section:GHPS_benchmark} with parameters \(\theta = 0.5\) and 
	\(\lambda_{\rm lin} = 0.7\).}
\end{figure}

\printbibliography

\appendix
\section{Proofs of Lemma~\ref{lemma:energy-identity}, 
Proposition~\ref{prop:aposteriori-jk},
Theorem~\ref{th:optimal_complexity}, and
Lemma~\ref{lem:estimator_equivalence}}
\label{appendix}

\begin{proof}[\textbf{Proof of Lemma~\ref{lemma:energy-identity}}]
	Let $d_\ell \coloneqq v_\ell-w_\ell$ and
	$\psi(t) \coloneqq \EE(w_\ell + t d_\ell)$. Arguing as in~\cite{hw2020},
	one sees that $\psi$ is differentiable with
	\begin{align*}
		&\dist(v_\ell, w_\ell)
		=
		\EE(w_\ell) - \EE(v_\ell)
		=
		\psi(0) - \psi(1)
		=
		- \int_0^1 \psi'(t) \d{t}
		\\&
		\quad
		= -
		\int_0^1
		\big[
			\dual{\AA(w_\ell + t d_\ell) }{d_\ell}_{\XX' \times \XX}
			- F(d_\ell)
		\big]
		\d{t}
		\\& \quad
		= -
		\Bigl(
			\int_0^1
			\dual{
				\AA(w_\ell + t d_\ell) - \AA u_\ell^{k-1,\jj}
				}{
				d_\ell
				}_{\XX' \times \XX}\d{t}
			-  F(d_\ell)
			+ \dual{\AA u_\ell^{k-1,\jj}}{d_\ell}_{\XX' \times \XX}
		\Bigr).
	\end{align*}
	Note that $u_\ell^{k,\star}$ is the exact iterate of the
	linearization~\eqref{eq:linearized_pb} and the latter yields that
	\begin{equation*}
		F(d_\ell) - \dual{ \AA u_\ell^{k-1,\jj}}{d_\ell}_{\XX' \times \XX}
		\eqreff{eq:linearized_pb}
		=
		a(u_\ell^{k-1,\jj}; u_\ell^{k,\star} - u_\ell^{k-1,\jj}, d_\ell).
	\end{equation*}
	To see~\eqref{eq2:lemma:energy-identity}, note that
	\begin{align*}
		\dual{
			\AA(w_\ell + t d_\ell) - \AA u_\ell^{k-1,\jj}
			}{
			d_\ell
			}_{\XX' \times \XX}
		&\le
		L \,
		\big[
			\enorm{w_\ell - u_\ell^{k-1,\jj}} + t \, \enorm{d_\ell}
		\big] \,
		\enorm{d_\ell}
		\\
		&=
		L \, \enorm{w_\ell - u_\ell^{k-1,\jj}} \,
		\enorm{d_\ell} + t L \, \enorm{d_\ell}^2.
	\end{align*}
	With $\int_0^1 t \d{t} = 1/2$, this concludes the proof
	of~\eqref{eq2:lemma:energy-identity}. For
	$w_\ell = u_\ell^{k-1,\jj}$, the last estimate simplifies to
	\begin{equation*}
		\dual{
				\AA(u_\ell^{k-1,\jj} + t d_\ell) - \AA u_\ell^{k-1,\jj}
			}{
				d_\ell
			}_{\XX' \times \XX}
		\le
		t L \, \enorm{d_\ell}^2
	\end{equation*}
	and analogously
	\begin{equation*}
		\dual{
				\AA(u_\ell^{k-1,\jj} + t d_\ell) - \AA u_\ell^{k-1,\jj}
			}{
				d_\ell
			}_{\XX' \times \XX}
		\ge
		t \alpha \, \enorm{d_\ell}^2.
	\end{equation*}
	Together with $\int_0^1 t \d{t} = 1/2$, this concludes the proof
	of~\eqref{eq2:lemma:energy-identity}--\eqref{eq3:lemma:energy-identity}.
\end{proof}

\begin{proof}[\textbf{Proof of Proposition~\ref{prop:aposteriori-jk}}]
	Reliability~\eqref{axiom:reliability} and stability~\eqref{axiom:stability}
	prove that
	\begin{align*}
		\enorm{u^\exact - u_\ell^{k,j}}
		\le
		\enorm{u^\exact - u_\ell^\exact} + \enorm{u_\ell^\exact - u_\ell^{k,j}}
		&\eqreff{axiom:reliability}
		\le
		\Crel \, \eta_\ell(u_\ell^\exact)
		+ \enorm{u_\ell^\exact - u_\ell^{k,j}}
		\\&
		\eqreff{axiom:stability}
		\le
		\Crel \, \eta_\ell(u_\ell^{k,j})
		+ (1 + \Crel \, \Cstab) \, \enorm{u_\ell^\exact - u_\ell^{k,j}}.
	\end{align*}
	The \textsl{a~posteriori} estimate~\eqref{eq2:alg:sol:contraction} bounds
	the final term on the right-hand side
	\begin{equation*}
		\enorm{u_\ell^\exact - u_\ell^{k,j}}
		\le
		\enorm{u_\ell^\exact - u_\ell^{k,\exact}}
		+ \enorm{u_\ell^{k,\exact} - u_\ell^{k,j}}
		\eqreff{eq2:alg:sol:contraction}
		\le
		\enorm{u_\ell^\exact - u_\ell^{k,\exact}}
		+ \frac{\qalg}{1-\qalg} \, \enorm{u_\ell^{k,j} - u_\ell^{k,j-1}}.
	\end{equation*}
	Overall, it only remains to bound
	$\enorm{u_\ell^\exact - u_\ell^{k,\exact}}$. From the unperturbed energy
	contraction~\eqref{eqq:energy-contraction} and
	using~\eqref{eq3:lemma:energy-identity} for $v_\ell = u_\ell^{k,\exact}$,
	it follows that
	\begin{align*}
		\frac{\alpha}{2} \, \enorm{u_\ell^\exact - u_\ell^{k,\exact}}^2
		&\eqreff{eq:energy}
		\le
		\dist(u_\ell^\exact, u_\ell^{k,\exact})
		\eqreff*{eqq:energy-contraction}
		\le
		\frac{\qnrgs}{1 - \qnrgs} \dist(u_\ell^{k,\exact} u_\ell^{k-1,\jj})
		\\&
		\eqreff*{eq3:lemma:energy-identity}
		\le
		\frac{\qnrgs}{1 - \qnrgs}
		\Bigl(
			- \frac{\alpha}{2} \, \enorm{u_\ell^{k,\exact}-u_\ell^{k-1,\jj}}^2
			+ a(
				u_\ell^{k-1,\jj}; u_\ell^{k,\star} - u_\ell^{k-1,\jj},
				u_\ell^{k,\exact}-u_\ell^{k-1,\jj}
				)
		\Bigr)
		\\&
		\eqreff*{eq:linearization:bounds}
		\le
		\frac{\qnrgs}{1 - \qnrgs}
		\Bigl( \Ccnt - \frac{\alpha}{2} \Bigr) \,
		\enorm{u_\ell^{k,\exact}-u_\ell^{k-1,\jj}}^2.
	\end{align*}
	Contraction of the algebraic solver~\eqref{eq:alg:sol:contraction} and
	nested iteration $u_\ell^{k,0} = u_\ell^{k-1,\jj}$ show that
	\begin{align*}
		\enorm{u_\ell^{k,\exact}-u_\ell^{k-1,\jj}}
		&\le
		\enorm{u_\ell^{k,\exact}-u_\ell^{k,j}}
		+ \enorm{u_\ell^{k,j}-u_\ell^{k-1,\jj}}
		\\&
		\eqreff*{eq:alg:sol:contraction}\le
		\qalg^j \, \enorm{u_\ell^{k,\exact}-u_\ell^{k-1,\jj}}
			+ \enorm{u_\ell^{k,j}-u_\ell^{k-1,\jj}}.
	\end{align*}
	Finally, we are thus led to
	\begin{equation*}
		\enorm{u_\ell^\exact - u_\ell^{k,\exact}}^2
		\le
		\frac{2}{\alpha} \frac{\qnrgs}{1-\qnrgs}
		\Big(\Ccnt - \frac{\alpha}{2} \Big)
		\frac{1}{(1-\qalg)^2} \, \enorm{u_\ell^{k,j}-u_\ell^{k-1,\jj}}^2.
	\end{equation*}
	This concludes the proof.
\end{proof}

\begin{proof}[\textbf{Proof of Theorem~\ref{th:optimal_complexity}}]
	Thanks to Corollary~\ref{corollary:rates:complexity}, it suffices to show
	\begin{equation}\label{eq:proof:optimality}
		\| u \|_{\A_s}
		\lesssim
		\sup_{(\ell, k, j) \in \QQ} (\# \TT_{\ell})^s \, \Eta_{\ell}^{k, j}
		\lesssim
		\max \{\| u \|_{\A_s}, \Eta_{0}^{0, 0}\}.
	\end{equation}

	\medskip
	\textbf{Step~1.} We first show the lower bound
	in~\eqref{eq:proof:optimality} for \(\elll = \infty\). In this case,
	Algorithm~\ref{algorithm} ensures that \(\# \TT_{\ell} \to \infty\) as
	\(\ell \to \infty\). For any \(N \in \N\), we follow the proof of
	\cite[Proposition~4.15]{axioms} and choose the maximal index
	\(\ell' \in \N_0\) with \(\# \TT_{\ell'} - \# \TT_{0} \le N\). Since NVB
	refinement gives
	\(\# \TT_{\ell'+1} \lesssim \# \TT_{\ell'}\), we are led to
	\begin{equation*}
		N + 1 < \# \TT_{\ell'+1} - \# \TT_{0} + 1
		\le
		\# \TT_{\ell'+1}
		\lesssim
		\# \TT_{\ell'}.
	\end{equation*}
	Then, using \(\TT_{\ell'} \in \T\) we
	show for all \((\ell', k', j') \in \QQ\) that
	\begin{equation*}
		\min_{\TT_{\rm opt} \in \T_N} \eta_{\rm opt}(u_{\rm opt}^\star)
		\le
		\eta_{\ell'}(u_{\ell'}^\star)
		\le  \Eta_{\ell'}^{k', j'}.
	\end{equation*}
	A combination of the two previous estimates verifies
	\begin{equation*}
		(N + 1)^s \,
		\min_{\TT_{\rm opt} \in \T_N} \eta_{\rm opt}(u_{\rm opt}^\star)
		\le
		(\# \TT_{\ell'})^s \, \Eta_{\ell'}^{k', j'}
		\le
		\sup_{(\ell, k, j) \in \QQ} (\# \TT_{\ell})^s \, \Eta_{\ell}^{k, j}.
	\end{equation*}
	Taking the supremum over all \(N \in \N\), we prove the lower
	bound in~\eqref{eq:proof:optimality} for \(\elll = \infty\).

	\medskip
	\textbf{Step~2.} We prove the lower bound in
	\eqref{eq:proof:optimality} for \(\elll < \infty\).
	Corollary~\ref{corollary_0tol} ensures \(\eta_\elll(u_\elll^\star) = 0\) and
	\(u_\elll^\star = u^\star\). Since \(\| u^\star \|_{\A_s} = 0\)
	for \(\elll = 0\), we may assume \(\elll > 0\). Following Step~1, we let
	\(0 \le N < \# \TT_{\elll} - \# \TT_{0} \) and choose the maximal index
	\(0 \le \ell' < \elll\) with \(\# \TT_{\ell'} - \# \TT_{0} \le N\). With
	the estimates from Step~1, we arrive at
	\begin{equation*}
		\| u^\star \|_{\A_s}
		=
		\sup_{0 \le N < \# \TT_{\elll} - \# \TT_{0}}
		\bigl(
			(N+1)^s
			\min_{\TT_{\rm opt} \in \T_N} \eta_{\rm opt}(u_{\mathrm{opt}}^\star)
		\bigr)
		\le
		\sup_{(\ell, k, j) \in \QQ} (\# \TT_{\ell})^s \, \Eta_{\ell}^{k, j}.
	\end{equation*}
	This proves the lower bound
	in~\eqref{eq:proof:optimality} also for \(\elll < \infty\).

	\medskip
	\textbf{Step~3.} We proceed to prove the upper bound
	in~\eqref{eq:proof:optimality}. To this end, we may suppose
	\(\| u^\star \|_{\A_s} < \infty\) since the result is trivial otherwise.
	Let $(\ell'+1, 0, 0) \in \QQ$. Then, \cite[Lemma~14]{ghps2021} guarantees
	the existence of a set \(\RRR_{\ell'} \subseteq \TT_{\ell'}\) such that
	\begin{equation}\label{eq:Rell}
		\# \RRR_{\ell'}
		\lesssim
		\| u^\star \|_{\A_s}^{1/s} \, \eta_{\ell'}(u_{\ell'}^\star)^{-1/s}
		\quad
		\text{and}
		\quad
		\thetamark \, \eta_{\ell'}(u_{\ell'}^\star)^2
		\le
		\eta_{\ell'}(\RRR_{\ell'}; u_{\ell'}^\star)^2.
	\end{equation}
	Note that~\eqref{eq:equivalence_Doerfler} verifies that \(\RRR_{\ell'}\)
	satisfies the Dörfler marking criterion in
	Algorithm~\ref{algorithm}\textrm{(iii)} with parameter \(\theta\), i.e.,
	\(
		\theta \, \eta_{\ell'}(u_{\ell'}^{\kk, \jj})^2
		\le
		\eta_{\ell'}(\RRR_{\ell'}; u_{\ell'}^{\kk, \jj})^2
	\).
	Therefore, since \(0 < \theta < \thetamark < \theta^\star\), the Dörfler
	marking criterion in~\eqref{eq:doerfler} ensures
	\begin{equation}\label{eq:Doerfler_opt}
		\#\MM_{\ell'}
		\le
		C_{\rm mark} \, \# \RRR_{\ell'}.
	\end{equation}
	Note that the algebraic error contraction~\eqref{eq:alg:sol:contraction} 
	and nested iteration $ u_{\ell'}^{\kk, 0} =  u_{\ell'}^{\kk-1, \jj}$ imply
		\begin{align*}
		\enorm{u_{\ell'}^{\kk, \star} - u_{\ell'}^{\kk, \jj}}
		\eqreff{eq:alg:sol:contraction} 
		\le 
		\qalg^{\jj} \, \enorm{u_{\ell'}^{\kk, \star} - u_{\ell'}^{\kk-1, \jj}}
		\le 
		\qalg \, \enorm{u_{\ell'}^{\kk, \star} - u_{\ell'}^{\kk, \jj}}
		+
		\qalg \, \enorm{u_{\ell'}^{\kk, \jj} - u_{\ell'}^{\kk-1, \jj}}
		\end{align*}
		which after rearranging the terms leads to
		\begin{align}\label{eq:erroralg}
		\enorm{u_{\ell'}^{\kk, \star} - u_{\ell'}^{\kk, \jj}}
		\le 
		\frac{\qalg}{1-\qalg} \, 
		\enorm{u_{\ell'}^{\kk, \jj} - u_{\ell'}^{\kk-1, \jj}}.
		\end{align}
	Full linear convergence~\eqref{eq:Rlin:convergence}, 
	stability~\eqref{axiom:stability} and triangle inequality, 
	contraction of the exact linearization~\eqref{eq:energy-contraction} and
	the stopping criterion of the linearization~(\hyperref[eq:st_lin]{Alg4}) show
	\begin{align}\label{eq:Eta_estimator}
		\begin{split}
			\Eta_{\ell'+1}^{0, 0}
			&\eqreff*{eq:Rlin:convergence}
			\lesssim
			\Eta_{\ell'}^{\kk, \jj}
			\eqreff{axiom:stability}
			\lesssim
			\enorm{u_{\ell'}^{\star} - u_{\ell'}^{\kk, \star}}
			+ \enorm{u_{\ell'}^{\kk, \star} - u_{\ell'}^{\kk, \jj}}
			+ \eta_{\ell'}(u_{\ell'}^{\kk, \jj})
			\\
			&
			\eqreff*{eq:energy}
			\lesssim
			\dist(u_{\ell'}^\star, u_{\ell'}^{\kk, \star})^{1/2}
			+ \enorm{u_{\ell'}^{\kk, \star} - u_{\ell'}^{\kk, \jj}}
			+ \eta_{\ell'}(u_{\ell'}^{\kk, \jj})
			\\
			&
			\eqreff*{eq:energy-contraction}
			\lesssim
			\dist(u_{\ell'}^\star, u_{\ell'}^{\kk-1, \jj})^{1/2}
			+ \enorm{u_{\ell'}^{\kk, \star} - u_{\ell'}^{\kk, \jj}}
			+ \eta_{\ell'}(u_{\ell'}^{\kk, \jj})
			\\
			&\eqreff*{eq:erroralg}
			\lesssim
			\dist(u_{\ell'}^\star, u_{\ell'}^{\kk-1, \jj})^{1/2}
			+ \enorm{u_{\ell'}^{\kk, \jj} - u_{\ell'}^{\kk-1, \jj}}
			+ \eta_{\ell'}(u_{\ell'}^{\kk, \jj})\\
			&\eqreff*{eq:norm_energy_estimate}
			\lesssim
			\dist(u_{\ell'}^\star, u_{\ell'}^{\kk-1, \jj})^{1/2}
			+ \dist(u_{\ell'}^{\kk, \jj}, u_{\ell'}^{\kk-1, \jj})^{1/2}
			+ \eta_{\ell'}(u_{\ell'}^{\kk, \jj})
			\\
			&
			\eqreff*{eq:pythagoras_discrete}
			\le
			\dist(u_{\ell'}^\star, u_{\ell'}^{\kk, \jj})^{1/2}
			+ \dist(u_{\ell'}^{\kk, \jj}, u_{\ell'}^{\kk-1, \jj})^{1/2}
			+ \eta_{\ell'}(u_{\ell'}^{\kk, \jj})
			\\
			&\eqreff*{prop:linearization:eq2+}
			\lesssim
			\dist(u_{\ell'}^{\kk, \jj}, u_{\ell'}^{\kk-1, \jj})^{1/2}
			+ \eta_{\ell'}(u_{\ell'}^{\kk, \jj})
			\stackrel{\textup{(\hyperref[eq:st_lin]{Alg4})}}
			\lesssim
			\eta_{\ell'}(u_{\ell'}^{\kk, \jj}).
		\end{split}
	\end{align}
	Combining all estimates and using the estimator
	equivalence~\eqref{eq:estimator_equivalence}, we obtain
	\begin{equation}\label{eq:marked_Eta}
		\medmuskip = -1mu
		\# \MM_{\ell'} \,
		\eqreff*{eq:Doerfler_opt}
		\lesssim \,
		\# \RRR_{\ell'} \,
		\eqreff*{eq:Rell}
		\lesssim \,
		\| u^\star \|_{\A_s}^{1/s} \, \eta_{\ell'}(u_{\ell'}^\star)^{-1/s}
		\eqreff*{eq:estimator_equivalence}
		\lesssim
		\| u^\star \|_{\A_s}^{1/s} \, \eta_{\ell'}(u_{\ell'}^{\kk, \jj})^{-1/s}
		\eqreff*{eq:Eta_estimator}
		\lesssim
		\| u^\star \|_{\A_s}^{1/s} \, \bigl(\Eta_{\ell'+1}^{0, 0}\bigr)^{-1/s} .
	\end{equation}

	\medskip
	\textbf{Step 4.} Consider $(\ell, k, j) \in \QQ$, then full linear
	convergence~\eqref{eq:Rlin:convergence} yields that
	\begin{equation}\label{eq:lin_cv_sum}
		\sum_{
			\substack{
				(\ell',k',j') \in \QQ
				\\
				|\ell',k',j'| \le |\ell,k,j|
			}
		}
		(\Eta_{\ell'}^{k', j'})^{-1/s}
		\eqreff{eq:Rlin:convergence}
		\lesssim
		(\Eta_{\ell}^{k, j})^{-1/s}
		\sum_{
			\substack{
				(\ell',k',j') \in \QQ
				\\
				|\ell',k',j'| \le |\ell,k,j|
			}
		}
		(\qlin^{1/s})^{|\ell, k, j| - |\ell', k', j'|}
		\lesssim
		(\Eta_{\ell}^{k, j})^{-1/s}.
	\end{equation}
	Recall that NVB refinement satisfies the mesh-closure estimate, i.e., there
	holds that
	\begin{equation}\label{eq:meshclosure}
		\# \TT_\ell - \# \TT_0
		\le
		\Cmesh \sum_{\ell' = 0}^{\ell-1} \# \MM_{\ell'}
		\quad \text{for all }
		\ell \ge 0,
	\end{equation}
	where $\Cmesh > 1$ depends only on $\TT_{0}$.
	Thus, for $(\ell, k, j) \in \QQ$ with $\ell > 0$, we have by the above
	formulas~\eqref{eq:marked_Eta}--\eqref{eq:meshclosure} that
	\begin{align*}
		\# \TT_{\ell} - \# \TT_{0} +1
		&\le
		2 \, (\# \TT_{\ell} - \# \TT_{0})
		\eqreff*{eq:meshclosure}
		\le
		2 \, \Cmesh \sum_{\ell' = 0}^{\ell-1} \# \MM_{\ell'}
		\eqreff*{eq:marked_Eta}
		\lesssim \
		\norm{u^\exact}_{\A_s}^{1/s}
		\sum_{\ell' = 0}^{\ell-1} (\Eta_{\ell'+1}^{0, 0})^{-1/s}
		\\
		&\lesssim
		\norm{u^\exact}_{\A_s}^{1/s} \! \!
		\sum_{
			\substack{
				(\ell',k',j') \in \QQ
				\\
				|\ell',k',j'| \le |\ell,k,j|
			}
		}
		(\Eta_{\ell'}^{k', j'})^{-1/s}
		\eqreff{eq:lin_cv_sum}
		\lesssim
		\norm{u^\exact}_{\A_s}^{1/s} (\Eta_{\ell}^{k, j})^{-1/s}.
	\end{align*}
	Rearranging the terms, we obtain that
	\begin{subequations}\label{eq:Eta_approxclass}
		\begin{equation}
			(\# \TT_{\ell} - \# \TT_{0} + 1)^s \, \Eta_{\ell}^{k, j}
			\lesssim
			\norm{u^\exact}_{\A_s}
			\quad \text{for all} \quad
			\ell > 0.
		\end{equation}
		Trivially, full linear convergence proves that
		\begin{equation}
			(\# \TT_{\ell} - \# \TT_{0} + 1)^s \, \Eta_{0}^{k, j}
			=
			\Eta_{0}^{k, j}
			\lesssim
			\Eta_{0}^{0, 0}
			\quad \text{for} \quad
			\ell = 0.
		\end{equation}
	\end{subequations}
	Moreover, recall from~\cite[Lemma~22]{bhp2017} that, for all
	$\TT_\coarse \in \T$ and all $\TT_\fine \in \T(\TT_\coarse)$,
	\begin{equation}\label{eq:bhp-lemma22}
		\# \TT_\fine - \# \TT_\coarse +1
		\le \# \TT_\fine
		\le  \# \TT_\coarse \, (\# \TT_\fine - \# \TT_\coarse +1).
	\end{equation}
	Overall, we have thus shown that
	\begin{equation*}
		(\# \TT_{\ell})^s \Eta_\ell^{k,j}
		\eqreff{eq:bhp-lemma22}
		\lesssim
		(\# \TT_{\ell} - \# \TT_{0} + 1)^s \Eta_\ell^{k,j}
		\eqreff{eq:Eta_approxclass}
		\lesssim
		\max \{\norm{u^\exact}_{\A_s}, \Eta_{0}^{0,0}\}
		\
		\text{for all $(\ell, k, j) \in \QQ$.}
	\end{equation*}
	This concludes the proof of the upper bound in~\eqref{eq:proof:optimality}
	and hence that of~\eqref{eq:optimal_complexity}.
\end{proof}

\begin{proof}[\textbf{Proof of Lemma~\ref{lem:estimator_equivalence}}]
	Norm equivalence~\eqref{eq:energy}, \emph{a~posteriori} error
	control~\eqref{prop:linearization:eq2+}, and the stopping
	criterion~(\hyperref[eq:st_lin]{Alg4}) allow to estimate
	\begin{align}
		\label{eq1:proof:estimator_equivalence}
		\begin{split}
			\enorm{u_\ell^\star& - u_\ell^{\kk, \jj}}
			\eqreff*{eq:energy}
			\le
			(2/\alpha)^{1/2} \,
			\dist(u_\ell^\star, u_\ell^{\kk, \jj})^{1/2}
			\eqreff*{prop:linearization:eq2+}
			\le
			(2/\alpha)^{1/2} \,
			\Big(\frac{\qnrg}{1-\qnrg}\Bigr)^{1/2}\,
			\dist(u_\ell^{\kk, \jj}, u_\ell^{\kk-1, \jj})^{1/2}
			\\
			&\stackrel{\textup{(\hyperref[eq:st_lin]{Alg4})}}
			\le
			(2/\alpha)^{1/2} \, \Big(\frac{\qnrg}{1-\qnrg}\Bigr)^{1/2}
			\,
			\lambda_{\rm lin}^{1/2}
			\,
			\eta_{\ell}(u_\ell^{\kk, \jj})
			\eqreff{eq:lambdalin}
			=
			C_{\textup{solve}} \,
			\lambda_{\rm lin}^{1/2} \, \eta_{\ell}(u_\ell^{\kk, \jj}).
		\end{split}
	\end{align}
	For any \(\UU_\ell \subseteq \TT_\ell\), stability~\eqref{axiom:stability}
	then allows to obtain
	\begin{alignat*}{2}
		\eta_{\ell}(\UU_\ell; u_\ell^{\kk, \jj})
		\eqreff*{axiom:stability}
		\le
		\eta_{\ell}(\UU_\ell; u_\ell^\star)
		+ \Cstab \, \enorm{u_\ell^\star - u_\ell^{\kk, \jj}}
		&\eqreff*{eq1:proof:estimator_equivalence}
		\le
		\
		&&\eta_{\ell}(\UU_\ell; u_\ell^\star)
		+ C_{\textup{solve}} \, \Cstab \, \lambda_{\rm lin}^{1/2} \,
		\eta_{\ell}(u_\ell^{\kk, \jj})
		\\
		&\le
		&&\eta_{\ell}(\UU_\ell; u_\ell^\star)
		+
		\bigl(\lambda_{\rm lin} / \lambda_{\rm lin}^\star\bigr)^{1/2} \,
		\eta_{\ell}(u_\ell^{\kk, \jj}).
	\end{alignat*}
	Then, the lower bound
	of~\eqref{eq:estimator_equivalence} follows by taking
	$\UU_\ell = \TT_{\ell}$, $0 < \lambda_{\rm lin} < \lambda_{\rm lin}^\star$,
	and rearranging the latter estimate. Similarly, it follows
	\begin{align}\label{eq2:proof:estimator_equivalence}
		\medmuskip = 2mu
		\begin{split}
			\eta_{\ell}(\UU_\ell; u_\ell^\star)
			\eqreff*{axiom:stability}
			\le
			\eta_{\ell}(\UU_\ell; u_\ell^{\kk, \jj})
			+ \Cstab \, \enorm{u_\ell^\star - u_\ell^{\kk, \jj}}
			&\eqreff*{eq1:proof:estimator_equivalence}
			\le \
			\eta_{\ell}(\UU_\ell; u_\ell^{\kk, \jj})
			+ \Cstab \, C_{\textup{solve}} \, \lambda_{\rm lin}^{1/2} \,
			\eta_{\ell}(u_\ell^{\kk, \jj})
			\\
			&\le \
			\eta_{\ell}(\UU_\ell; u_\ell^{\kk, \jj})
			+
			\bigl(\lambda_{\rm lin} / \lambda_{\rm lin}^\star \bigr)^{1/2} \,
			\eta_{\ell}(u_\ell^{\kk, \jj}).
		\end{split}
	\end{align}
	For \(\UU_\ell = \TT_{\ell}\), this concludes the proof
	of~\eqref{eq:estimator_equivalence}. Suppose
	\(\RRR_{\ell} \subseteq \TT_\ell\)
	satisfies
	\(
		\thetamark^{1/2} \, \eta_{\ell}(u_\ell^\star)
		\le
		\eta_{\ell}(\RRR_{\ell}; u_\ell^\star)
	\).
	Arguing as for~\eqref{eq:estimator_equivalence}, we see
	\begin{align*}
		\bigl[
			1
			- \bigl(\lambda_{\rm lin} / \lambda_{\rm lin}^\star\bigr)^{1/2}
		\bigr] \,
		\thetamark^{1/2}
		&\eta_{\ell}( u_\ell^{\kk, \jj})
		\eqreff*{eq:estimator_equivalence}
		\le
		\thetamark^{1/2} \, \eta_{\ell}(u_\ell^\star)
		\le
		\eta_{\ell}(\RRR_{\ell}; u_\ell^\star)
		\\
		&\eqreff*{eq2:proof:estimator_equivalence}
		\le
		\eta_{\ell}(\RRR_{\ell}; u_\ell^{\kk, \jj})
		+ \bigl(\lambda_{\rm lin} / \lambda_{\rm lin}^\star\bigr)^{1/2} \,
		\eta_{\ell}(u_\ell^{\kk, \jj})
		\\
		&
		\eqreff*{eq:thetamark}
		=
		\eta_{\ell}(\RRR_{\ell}; u_\ell^{\kk, \jj})
		+
		\bigl\{
			\thetamark^{1/2} \,
			\bigl[
				1
				- \bigl(\lambda_{\rm lin} / \lambda_{\rm lin}^\star\bigr)^{1/2}
			\bigl]
			- \theta^{1/2}
		\bigr\} \,
		\eta_{\ell}(u_\ell^{\kk, \jj}).
	\end{align*}
	This yields
	\(
		\theta \, \eta_{\ell}(u_\ell^{\kk, \jj})^2
		\le
		\eta_{\ell}(\RRR_{\ell}; u_\ell^{\kk, \jj})^2
	\)
	and concludes the proof.
\end{proof}

\end{document}